\documentclass[10pt]{article}
\usepackage[a4paper,
            left=1in,
            right=1in,
            top=1in,
            bottom=1in,
            footskip=.25in
]{geometry}

\usepackage[sc]{titlesec}
\usepackage{blindtext,color}
\usepackage{microtype}
\usepackage{scrextend}
\usepackage{upgreek}
\usepackage{stmaryrd}
\usepackage{subcaption}
\usepackage{tocloft}
\usepackage{algorithm}
\usepackage[noend]{algpseudocode}
\usepackage{tcolorbox}
\usepackage{amssymb}
\usepackage{amsmath}
\usepackage{amsthm}
\usepackage{amsfonts}
\usepackage{shadethm}
\usepackage{bbold}
\usepackage{xspace}
\usepackage{dsfont}
\usepackage{marginnote}
\usepackage{bm}
\usepackage{graphicx}
\usepackage[nottoc]{tocbibind}
\usepackage[citecolor=blue,colorlinks=true,linkcolor=blue]{hyperref}%
\usepackage{cleveref}
\usepackage{booktabs,multirow}
\usepackage{csvsimple}
\usepackage{tikz}
\usetikzlibrary{cd}
\usepackage{sectsty}

\title{Preconditioner Design via the Bregman Divergence}
\author{Andreas Bock \& Martin S. Andersen}
\date{}
\newcounter{mythm}
\newcounter{mylem}
\newcounter{myrem}

\newcounter{mypro}
\newcounter{mycol}

\newcounter{mydef}

\newtheorem{definition}[mydef]{Definition}
\newtheorem{remark}[myrem]{Remark}
\newtheorem{lemma}[mylem]{Lemma}

\newtheorem{proposition}[mypro]{Proposition}
\newtheorem{corollary}[mycol]{Corollary}

\newcommand{\Divergence}{\mathcal{D}}
\newcommand{\BregmanLogDet}{\mathcal{D}_\textnormal{LD}}
\newcommand{\vonNeumann}{\mathcal{D}_\textnormal{VN}}
\newcommand{\Frobenius}{\mathcal{D}_\textnormal{F}}

\newcommand{\blkdiag}{\textnormal{\texttt{blkdiag}}}

\newcommand{\expect}[1]{\mathds{E}\left[ #1 \right]}

\newcommand{\inv}{^{-1}}
\newcommand{\invhalf}{^{-\half}}
\newcommand{\transp}{^\top}
\newcommand{\invtransp}{^{-\top}}
\newcommand{\Hermitian}{^*}
\newcommand{\invHermitian}{^{-*}}

\newcommand{\invhalfHermitian}{^{{-\halfHermitian}}}

\newcommand{\trace}[1]{\Tr\big( #1 \big)}
\newcommand{\logdet}[1]{\log \det( #1 )}
\newcommand{\half}{{\frac{1}{2}}}

\newcommand{\halfHermitian}{{\frac{*}{2}}}

\newcommand{\Cn}{\mathbb{C}^n}
\newcommand{\Sn}{\mathbb{H}^n}
\newcommand{\Sr}{\mathbb{H}^r}
\newcommand{\Snp}{{\Sn_+}}
\newcommand{\Snpp}{{\Sn_{++}}}
\newcommand{\Srpp}{{\Sr_{++}}}

\newcommand{\bH}{\mathbf{H}}

\newcommand{\bL}{\mathbf{L}}
\newcommand{\bR}{\mathbf{R}}
\newcommand{\bS}{\mathbf{S}}
\newcommand{\bX}{\mathbf{X}}
\newcommand{\bK}{\mathbf{K}}
\newcommand{\bQ}{\mathbf{Q}}
\newcommand{\bG}{\mathbf{G}}
\newcommand{\bdx}{\mathbf{\delta x}}
\newcommand{\bb}{\mathbf{b}}
\newcommand{\bd}{\mathbf{d}}

\newcommand{\Normal}{\mathcal{N}}
\newcommand{\diff}{\,\textnormal{d}}

\newcommand{\Rm}{\mathbb{R}^m}
\newcommand{\rn}{\mathbb{R}^n}

\newcommand{\bD}{\mathbf{D}}
\newcommand{\bDh}{\mathbf{D}^\half}
\newcommand{\bDi}{\mathbf{D}^{-1}}

\newcommand{\Rnn}{\mathbb{R}^{n\times n}}
\newcommand{\Rmm}{\mathbb{R}^{m\times m}}
\newcommand{\Rnm}{\mathbb{R}^{n\times m}}

\newcommand{\Rnr}{\mathbb{R}^{n\times r}}
\newcommand{\Rrr}{\mathbb{R}^{r\times r}}
\newcommand{\Rsr}{\mathbb{R}^{s\times r}}

\newcommand{\Cnn}{\mathbb{C}^{n\times n}}

\newcommand{\Cnr}{\mathbb{C}^{n\times r}}
\newcommand{\Crr}{\mathbb{C}^{r\times r}}

\newcommand{\MatVec}{\pi}

\DeclareMathOperator{\diag}{diag}
\DeclareMathOperator{\Tr}{trace}
\DeclareMathOperator{\rank}{rank}
\DeclareMathOperator{\range}{range}
\DeclareMathOperator{\interior}{int}
\DeclareMathOperator{\domain}{dom}
\DeclareMathOperator{\relint}{ri}
\DeclareMathOperator{\aff}{aff}
\DeclareMathOperator*{\minimise}{minimise}
\DeclareMathOperator*{\subto}{s.t.}

\newcommand{\randsvd}[1]{{#1}\langle\Omega\rangle}
\newcommand{\randsvdpower}[1]{\randsvd{{#1}^q}}
\newcommand{\Nystrom}[1]{{#1}^\text{Nys}\langle\Omega\rangle}
\newcommand{\NystromQ}[1]{{#1}^\text{Nys}\langle\Theta\rangle}

\newcommand{\nonscaled}[1]{\tilde{#1}_r}
\newcommand{\nonscaledrand}[1]{\randsvd{\tilde{#1}}}
\newcommand{\nonscaledrandpower}[1]{\randsvdpower{\tilde{#1}}}
\newcommand{\nonscaledNys}[1]{\Nystrom{\tilde{#1}}}
\newcommand{\nonscaledNysQ}[1]{\NystromQ{\tilde{#1}}}

\newcommand{\scaled}[1]{\widehat{#1}_r}
\newcommand{\scaledrand}[1]{\randsvd{\widehat{#1}}}
\newcommand{\scaledrandpower}[1]{\randsvdpower{\widehat{#1}}}
\newcommand{\scaledNys}[1]{\Nystrom{\widehat{#1}}}
\newcommand{\scaledNysQ}[1]{\NystromQ{\widehat{#1}}}

\newtheorem{theorem}[mythm]{Theorem}

\begin{document}
\maketitle
\begin{abstract}
\noindent
We study a preconditioner for a Hermitian positive definite linear system, which
is obtained as the solution of a matrix nearness
problem based on the Bregman log determinant divergence. The preconditioner
is of the form of a Hermitian positive definite matrix plus a low-rank matrix.
For this choice of structure, the generalised eigenvalues of the 
preconditioned matrix are easily calculated, and we show under which conditions
the preconditioner minimises the $\ell_2$ condition number
of the preconditioned matrix. We develop practical numerical approximations of
the preconditioner based on the randomised singular value decomposition
(SVD) and the Nystr\"om approximation and provide corresponding approximation
results. Furthermore, we prove that the Nystr\"om approximation is in fact also a
matrix approximation in a range-restricted Bregman divergence and establish
several connections between this divergence and matrix nearness problems
in different measures. Numerical examples are provided to support the
theoretical results.
\end{abstract}
\section{Introduction}
We study preconditioning of a Hermitian matrix $S = A + B \in\Cnn$, 
where $A=QQ\Hermitian \in\Cnn$ is positive definite and $B \in\Cnn$ is 
Hermitian positive semidefinite. The factor $Q \in\Cnn$
does not need to be a Cholesky factor and can be a symmetric square root,
for instance. Finding $x\in\Cn$ such that
\begin{equation}\label{eq:Sxb}
Sx = b,
\end{equation}
encapsulates regularised least squares, Gaussian process regression 
\cite{rasmussen2006gaussian} and is a ubiquitous problem in scientific computing
\cite{zhang2013constraint,benzi2005numerical}. It also appears in Schur complements
of saddle-point formulations of variational data assimilation problems 
\cite{freitag2018low,freitag2020numerical}.
It is often the case in large-scale numerical linear algebra that interaction with
a matrix is only feasible through its action on vectors. Iterative methods such as
the conjugate gradient method are therefore of interest for which preconditioning 
is often a necessity for an efficient and accurate solution, leading to the so-called
\emph{preconditioned} conjugate gradient method (PCG). In this paper, we 
investigate preconditioners for \eqref{eq:Sxb} based on low-rank approximations
of the positive semidefinite term of $S$. We can write $S$ as
\[
S = Q(I + G)Q\Hermitian,
\]
where $G = Q\inv B Q\invHermitian$. Based on this decomposition we study the two
following preconditioners
\begin{subequations}\label{eq:S_preconditioners}
\begin{align}
& \nonscaled{S} =  A + B_r,\label{eq:S_preconditioners:nonscaled}\\
& \scaled{S} = Q(I + G_r)Q\Hermitian,\label{eq:S_preconditioners:scaled}
\end{align}
\end{subequations}
where $B_r$ and $G_r$ are both obtained as truncated SVDs of $B$ and $G$, respectively,
where $r<\rank(B)$.
While \eqref{eq:S_preconditioners:nonscaled} may seem natural, we show that our
proposed preconditioner \eqref{eq:S_preconditioners:scaled} is always an improvement
and has many useful properties from a practical and theoretical point
of view. We also demonstrate that it appears naturally as the minimiser of a matrix
nearness problem in the Bregman divergence. This is intuitive since it is
desirable to seek a preconditioner that is, in some sense, close to the matrix
$S$.\\

Our main contribution is the development of a general framework for the design
of preconditioners for \eqref{eq:Sxb} based on the Bregman log determinant 
divergence. First,
Section \ref{sec:spectral_analysis} examines the two preconditioners in
\eqref{eq:S_preconditioners} more closely based on straightforward eigenvalue analysis
and presents a simple example to develop intuition. We then introduce the Bregman
divergence framework in Section \ref{sec:precond_bregman}.
We prove that our proposed preconditioner \eqref{eq:S_preconditioners:scaled}
minimises the Bregman divergence to the original matrix $S$
in Theorem \ref{thm:whSS_minimiser} (and in some cases, the condition number of
the preconditioned  matrix), and demonstrate how it can improve the convergence
of PCG in Section \ref{sec:iterative_solution_methods}. In Section
\ref{sec:matrix_nearness_problems}, we make a connection between
different norm minimisation problems and the Bregman divergence
framework. We also derive the Nystr\"om approximation in several different ways,
one of which is as a minimiser of a range-restricted Bregman divergence. Next,
Section \ref{sec:rsvd} explores practical numerical methods based
on randomised linear algebra for the approximations of our preconditioner in the big data 
regime where the truncated SVD becomes computationally intractable. Section \ref{sec:numerical_results} supports
our theoretical findings with numerical experiments for equation \eqref{eq:Sxb} using 
various preconditioners. Section \ref{sec:summary} contains a summary.

\subsection{Related Work}

Preconditioning and matrix approximation are
closely linked since a preconditioner $\mathcal{P}$ is typically chosen such that, in 
some sense, $\mathcal{P} \approx S$, subject to the requirement that 
$\mathcal{P}$ respects some computational budget in terms of storage, action on 
vectors, and its factorisation. In this paper, we establish a natural connection between 
matrix nearness problems
in the Bregman divergence and preconditioning. The use of divergences in linear algebra
is not novel. \cite{dhillon2008matrix} proposed the use of the Bregman divergences for 
matrix approximation problems motivated by their connection to exponential families
of distributions. Also relevant to our work is \cite{kulis2009low} where the authors
investigate extensions of divergences to the low-rank matrices in the context of
kernel learning.
We build on these ideas in Section \ref{sec:matrix_nearness_problems} and show the
connection between the Bregman divergence and the Nystr\"om approximation.
The latter is becoming increasingly popular in numerical linear algebra,
in particular in the context of randomised matrix approximation methods.
Many such methods
are based on \emph{sketching} \cite{woodruff2014sketching}, a technique whereby a matrix
$A\in\Cnn$ is multiplied by some \emph{sketching matrix}
$\Omega\in\Rnr$, $r<n$ producing a compressed matrix $A\Omega$ which is used to
compute approximations of $A$. As mentioned, randomised methods select $\Omega$ as a
random matrix providing the setting for probabilistic bounds on the accuracy of the
approximation, see \cite{martinsson2020randomized,halko2011finding}.
Methods such as the randomised SVD and the Nystr\"om approximation offer
practical alternatives to the truncated SVD with both affordable computational cost and 
certain theoretical guarantees. Such approximations have long been popular in the 
kernel-based learning community \cite{williams2000using,drineas2005nystrom}. Preconditioning
is becoming more relevant for the data science and machine learning communities
with the increasing demand for handling large-scale problems. Recently, the work of
\cite{frangella2021randomized} introduced a Nystr\"om-based preconditioner for PCG for
matrices $\mu I + A$ for some scalar $\mu>0$, where $A$ is symmetric positive semidefinite.
This strategy has also been applied to the \emph{alternating method of multipliers}
\cite{zhao2022nysadmm}, and sketching more generally to a stochastic quasi-Newton method
in \cite{frangella2022sketchysgd}. Applications such as large-scale
Gaussian process regression can also benefit from the development of the preconditioners
introduced in this work since gradient-based parameter optimisation requires expensive
linear solves for optimising kernel hyperparameters \cite{wenger2022preconditioning}.

\subsection{Notation \& Preliminaries}

\begin{definition}
Let $\Sn$ denote the space of Hermitian $n\times n$ matrices. $\Snp \subset \Sn$ denotes
the cone of positive semidefinite matrices and $\Snpp = \interior \Snp$ the positive definite cone.
\end{definition}

\begin{definition}
Let $\lambda(A) = \{\lambda_1(A),\ldots,\lambda_n(A)\}$, with $\lambda_1(A) \geq
\ldots \geq \lambda_n(A)$, denote the 
ordered eigenvalues of a matrix $A\in\Cnn$
Further, let $\sigma(A) = \{\sigma_1(A),\ldots,\sigma_n(A)\}$ denote the 
similarly ordered singular values of a matrix $A\in\Cnn$. For nonsingular normal
matrices $A$ we denote the $\ell_2$ condition number by $\kappa_2(A)
= \frac{\sigma_1(A)}{\sigma_n(A)}$.
\end{definition}

We denote by $A:B :=\trace{A\Hermitian B}$ the trace inner product between two matrices
$A$ and $B$. $I$ denotes the identity matrix, sometimes with a subscript to denote the
dimension if it is not clear from the context.
$\|\cdot\|_F$ denotes the Frobenius norm and $\|\cdot\|_2$ the $\ell_2$ 
norm. $A^+$ denotes the Moore-Penrose inverse of $A$. $A_r$ will in general be used
to denote a rank $r$ approximation of $A$ obtained via a truncated SVD. When $A$ is
positive definite we define the induced norm $\|\cdot\|_A$ by the relation
\[
\|x\|_A^2 = x\Hermitian A x.
\]

Convex analysis is a natural tool for studying the Bregman divergence, so we
recall some elementary definitions, see \cite{rockafellar1997convex} for
more details. We restrict our attention to spaces of finite dimension. The \emph{relative interior}, $\relint C$, of a convex set $C$ is the 
subset of $C$ which can be considered as an affine subset of $C$, i.e.,
\[
\relint C = \{ x\in C \,|\, \exists\; \textnormal{neighbourhood}\;V  \textnormal{of}\;x \; \textnormal{such that}\;V\cap \aff C \subseteq C \},
\]
where $\aff C$ is the affine hull of $C$.
Let $\mathcal{X}$ be some finite-dimensional real inner product space.
The \emph{effective domain} of a function $\phi: \mathcal{X} \rightarrow
\mathbb{R} \cup \{-\infty,+\infty\}$ is defined by
\[
\domain \phi = \{ x \in \mathcal{X}\,|\, \phi(x) < \infty \}.
\]
A convex function $\phi$ is \emph{proper} if it never attains the value $-\infty$
and that there exists a point $x\in\mathcal{X}$ for which $\phi(x)$ is finite.

\section{Spectral and Other Properties}\label{sec:spectral_analysis}

In this section, we develop some intuition about the two preconditioners $\nonscaled{S}$ and $\scaled{S}$ defined in 
\eqref{eq:S_preconditioners} by straightforward analysis
of the eigenvalues of the preconditioned matrices $\nonscaled{S}\inv S$
and $\scaled{S}\inv S$. Their respective inverses are given by
\begin{align*}
\nonscaled{S}\inv & = Q\invHermitian ( I + Q\inv B_r Q\invHermitian)\inv Q\inv,\\
\scaled{S}\inv & = Q\invHermitian ( I + G_r)\inv Q\inv.
\end{align*}

\begin{lemma}\label{lemma:eigenvalues_preconditioned_matrices}
The eigenvalues of $\nonscaled{S}\inv S$ and $\scaled{S}\inv S$ satisfy the following
bounds for $i=1,\ldots,n$:
\begin{align*}
& \lambda_i(\nonscaled{S}\inv S) \in \left[1, 1 + \lambda_1(G)\right],\\
& \lambda_i(\scaled{S}\inv S) \in \left[1, 1 + \lambda_{r+1}(G)\right].
\end{align*}
\end{lemma}
\begin{proof}
Recall that $B_r$ and $G_r$ are obtained as truncated SVDs of $B$ and $G$, respectively,
where $r<\rank(B)$. We obtain the result by observing the following bounds for the Rayleigh
quotients below, where $x\in\mathbb{C}^n$ and not identically zero:
\begin{subequations}\label{eq:Rayleigh}
\begin{align}
\frac{x\Hermitian (I + G) x}{x\Hermitian (I + Q\inv B_r Q\invHermitian) x} &\in \left[1, 1 + \lambda_1(G)\right],\label{eq:Rayleigh:nonscaled}\\
 \frac{x\Hermitian (I + G) x}{x\Hermitian (I + G_r) x} &\in \left[1, 1 + \lambda_{r+1}(G)\right].\label{eq:Rayleigh:scaled}
\end{align}
\end{subequations}
\end{proof}

The bound \eqref{eq:Rayleigh:nonscaled} is of course pessimistic, 
but it is in general difficult to infer a tighter upper bound here
since $G$ and $Q\inv B_r Q\invHermitian$ are not a priori simultaneously diagonalisable.
Later, in Theorem \ref{thm:kappa2_minimisation}, we show that when $G$ does not have
full rank (but not necessarily low rank), $\scaled{S}$ minimises
the conditioner number of the preconditioned matrix among matrices of the form $A + W$,
where $W\in\Snp$ has at rank at most $r$. We defer a more detailed analysis to Section
\ref{sec:precond_bregman}, but this result motivates the study of
\eqref{eq:S_preconditioners}.\\

Next we present a simple example illustrating
the difference between
$\scaled{S}$ and $\nonscaled{S}$. Let $A$ and $B$ be diagonal $6\times 6$ matrices:
\begin{equation}\label{eq:AB_diagonal_example}
A = \diag(1.1, 1.05, 0.375, 0.05, 0.05, 0.05), \qquad B = \diag(1, 0.5, 0.25, 0.1, 0, 0).
\end{equation}
This results in
\begin{subequations}
\begin{align}
& G = \diag(\mathbf{0.9091}, 0.4762, 0.6667, \mathbf{2}, 0, 0),\label{eq:G_diagonal_example}\\
& Q = \diag(1.0488, 1.0247, 0.6124, 0.2236, 0.2236, 0.2236),
\end{align}
\end{subequations}
where, for ease of exposition, we truncate after the fourth decimal.
We observe $\rank(B) = 4$, and take $r=2$ in the low-rank approximations of the
positive semidefinite term. A truncated SVD will pick out the indices corresponding
to the $r$ largest numbers (in bold). On the other hand, we find that
\[
Q\inv B_r Q\invHermitian = \diag(\mathbf{0.9091}, \mathbf{0.4762}, 0, 0, 0, 0).
\]

The resulting generalised eigenvalues are shown in 
Figure \ref{fig:AB_diagonal_example}.
\begin{figure}
    \centering
    \includegraphics[scale=0.35
    ]{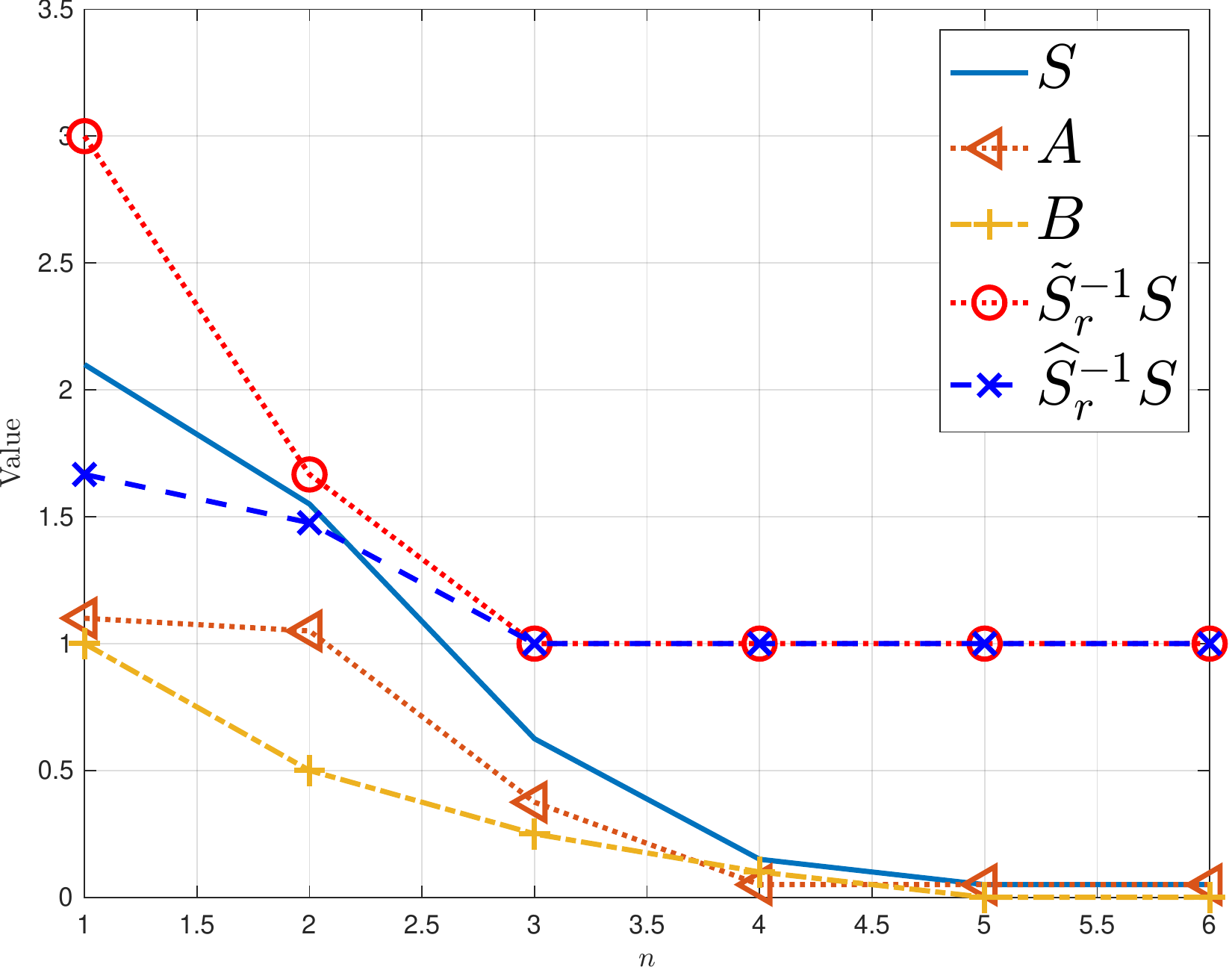}
    \caption{Eigenvalues of $S$ and generalised eigenvalues of the preconditioned system
    using the preconditioners $\nonscaled{S}$ and $\scaled{S}$.}
    \label{fig:AB_diagonal_example}
\end{figure}
Recall that
\begin{align*}
\nonscaled{S}\inv S & = Q\invHermitian (I + Q\inv B_r Q\invHermitian)\inv (I + G) Q\Hermitian,\\
\scaled{S}\inv S & = Q\invHermitian ( I + G_r)\inv (I + G) Q\Hermitian.
\end{align*}
For the instances of $A$ and $B$ in \eqref{eq:AB_diagonal_example} we therefore 
have, by direct calculation,
\[
\lambda_i(Q\inv B_r Q\invHermitian) = \lambda_i(B) \lambda_i(A)\inv, \quad i=1,\ldots, r,
\]
and
\[
\lambda_1(G_r) = B_{44} A_{44}\inv, \quad \lambda_2(G_r) = \lambda_1(B)\lambda_1(A)\inv,
\]
which clearly shows the scaling effect. In particular, we note that
\[
\lambda_2(G_r) = \lambda_1(Q\inv B_r Q\invHermitian).
\]
There is no difference when the spectrum
of $A$ is flat, but if the spectrum of $A$ decays as above, then 
$\lambda_i(Q\inv B_r Q\invHermitian)$ will be dominated by $\lambda_i(G_r)$ in which case:
\[
1 \leq \lambda_i(\scaled{S}\inv S) < \lambda_i(\nonscaled{S}\inv S), \qquad i=1,\ldots, n.
\]
We highlight a situation where $\scaled{S} = \nonscaled{S}$. Suppose that we had defined
$B$ by reversing the order of the non-zero diagonal elements in \eqref{eq:AB_diagonal_example},
i.e.,
\[
B^\textnormal{rev} = \diag(0.1, 0.25, 0.5, 1, 0, 0).
\]
In this case,
\begin{equation}\label{eq:G_diagonal_example_reversed}
G^\textnormal{rev} = \diag(1.1\inv, 0.2381, \mathbf{1.3333}, \mathbf{20}, 0, 0),
\end{equation}
and
\[
Q\inv B_r^\textnormal{rev} Q\inv = \diag(0, 0, \mathbf{1.3333}, \mathbf{20}, 0, 0) = G_r^\textnormal{rev}.
\]
By comparing \eqref{eq:G_diagonal_example} and \eqref{eq:G_diagonal_example_reversed},
we notice the largest eigenvalues (in bold) have different indices, representing the loss of
information when prematurely truncating $B$ before scaling by the factors of $A\inv$.
This suggests that this scaling can, in some circumstances, improve the approximation
of $S$. While this approach uses $G$ explicitly, we show in Section \ref{sec:rsvd} 
economical ways of computing with it using only matrix-vector products.
In the general case where the basis is not given by the identity, analysis of the
generalised eigenvalues is of course insufficient to determine the quality of the
associated preconditioners, a point to which we return later.

\begin{corollary}\label{cor:scaled_equals_nonscaled}
$\scaled{S} = \nonscaled{S}$ when $A$ is a scaled identity.
\end{corollary}
\begin{proof}
Let $A = \alpha I$ for some $\alpha > 0$. Then,
\[
G_r = (\alpha^{-1/2} B \alpha^{-1/2})_r =  \alpha\inv B_r = Q\inv B_r Q\invHermitian.
\]
\end{proof}

\section{Preconditioners via Bregman Projections}\label{sec:precond_bregman}

In this section we analyse the problem of finding a preconditioner by
using the Bregman divergence. We will formulate this as a constrained 
optimisation problem and prove \eqref{eq:S_preconditioners:scaled} is a 
minimiser. We also present a result concerning the optimality of the
condition number of the resulting preconditioned matrix.
Section \ref{sec:splitting} discusses extentions of this framework,
and in Section \ref{sec:diag_pcs} we give an example of how certain
diagonal preconditioners can be derived from the proposed framework.
Finally, Section \ref{sec:iterative_solution_methods} summarises the
consequences of our theoretical results in the context of PCG.\\

It is well-known that a truncated singular value decomposition
is an optimal low-rank approximation in both the spectral and Frobenius norms thanks to the 
Eckart–Young–Mirsky theorem. These norms are not invariant under
congruence transformations, i.e., $Z\mapsto P\Hermitian Z P$ for some
invertible $P$, and we will discover in Section \ref{sec:matrix_nearness_problems} that the preconditioners
$\nonscaled{S}$ and $\scaled{S}$ are optimal solutions to different matrix
norm minimisation problems. In this section, we introduce a different nearness
measure, the Bregman divergence, which naturally leads to the preconditioner 
$\scaled{S}$.\\

The Bregman matrix divergence $\Divergence_\phi:\domain \phi\times\relint \domain\phi\rightarrow [0, \infty)$
associated with a proper, continuously-differentiable, strictly convex 
\emph{seed} function $\phi$ is defined as follows:
\begin{align*}
    \Divergence_\phi(X,Y) & = \phi(X) - \phi(Y) - \trace{\nabla\phi(Y)\Hermitian (X-Y)}.
\end{align*}
This divergence originated in convex analysis for finding intersections of convex
sets \cite{bregman1967relaxation}, and divergences are now central objects in the
field of information geometry, see e.g. \cite{amari2016information,amari2010information}.
The choices $\phi(X) = \trace{X\log X - X}$, $\phi(X) = \|X\|_F^2$ and $\phi(X) = 
-\log\det(X)$ lead to the \emph{von Neumann}, \emph{squared Frobenius} and \emph{log
determinant} (or \emph{Burg}) divergences, respectively, where $\log$ in this
context is the matrix logarithm:
\begin{subequations}
\[
 \begin{alignedat}{2}
& \vonNeumann(X, Y) = \trace{X\log X - X\log Y -X + Y}, \qquad && \domain\phi = \Snpp,\\
& \Frobenius(X, Y) = \| X - Y\|_F^2, && \domain\phi = \Sn,\\
& \BregmanLogDet(X, Y) = \trace{XY\inv} - \logdet{XY\inv} - n, && \domain\phi = \Snpp.
 \end{alignedat}
\]
\end{subequations}
By a limit argument \cite[Section 4]{kulis2009low}, the log determinant matrix divergence
is only finite if $\range X = \range Y$, or $\range X \subseteq \range Y$ for the von
Neumann divergence. We highlight a few useful properties of the divergences, collectively
denoted by $\Divergence_\phi$ \cite{dhillon2008matrix}:
\begin{itemize}
    \item $\Divergence_\phi(X, Y) = 0 \Leftrightarrow X=Y$,
    \item \emph{Nonnegativity}: $\Divergence_\phi(X, Y) \geq 0$,
    \item \emph{Convexity}: $X\rightarrow \Divergence_\phi(X, Y)$ is convex.
\end{itemize}
We also state the following facts about these divergences:
\begin{corollary}[{\cite[Corollary 2]{kulis2009low}}]\label{cor:divergence_scalar_eigs}
Given eigendecompositions 
\begin{align*}
& X=U\diag(\lambda) U\Hermitian,\\
& Y=V\diag(\theta) V\Hermitian,
\end{align*}
with unitary $U$ and $V$ whose columns consist of eigenvectors $u_i$ and $v_i$, respectively,
the squared Frobenius, von Neumann and log determinant matrix divergences satisfy:
\begin{subequations}\label{eq:divergence}
\begin{align}
& \BregmanLogDet(X, Y) = \sum_{i=1}^n \sum_{j=1}^n (v_i\Hermitian u_j)^2 \left[\frac{\lambda_i}{\theta_j} - \log\frac{\lambda_i}{\theta_j} - 1\right],\\
& \vonNeumann(X, Y) = \sum_{i=1}^n \sum_{j=1}^n (v_i\Hermitian u_j)^2 [\lambda_i \log\lambda_i - \lambda_i\log\theta_j - \lambda_i + \theta_j],\\
& \Frobenius(X, Y) = \sum_{i=1}^n \sum_{j=1}^n (v_i\Hermitian u_j)^2 [\lambda_i - \theta_j]^2.
\end{align}
\end{subequations}
\end{corollary}

\begin{proposition}[{\cite[Proposition 12]{kulis2009low}}]\label{prop:bregman_invariance}
Let $P$ be invertible so $Z\mapsto f(Z)=P\Hermitian Z P$ is
a congruence transformation. Then,
\[
\BregmanLogDet(X, Y) = \BregmanLogDet(f(X), f(Y)).
\]
\end{proposition}

From Proposition \ref{prop:bregman_invariance}, we have
\begin{align}
\BregmanLogDet(\scaled{S}, S)
& = \BregmanLogDet(Q(I + G_r)Q\Hermitian, Q(I + Q\inv B Q\invHermitian)Q\Hermitian)\nonumber\\
& = \BregmanLogDet(I + G_r, I + Q\inv B Q\invHermitian).\label{eq:scaled_invariance_property}
\end{align}
In light of this, we ask if $G_r$ is in fact a solution
to the problem
\begin{subequations}\label{eq:whSS}
\begin{align}
\minimise_{X\in\Snp}\quad & \BregmanLogDet(\mathcal{P}, S)\\
\subto\quad
& \mathcal{P} = Q(I + X)Q\Hermitian\label{eq:P_choice}\\
& \rank(X) \leq r\label{eq:rank_of_X},
\end{align}
\end{subequations}
or, equivalently,
\begin{align*}
\minimise_{X\in\Snp}\quad & \BregmanLogDet(I + X, I + Q\inv B Q\invHermitian)\\
\subto\quad & \rank(X) \leq r.
\end{align*}
We answer this question in the affirmative with the following Theorem.
\begin{theorem}\label{thm:whSS_minimiser}
$\scaled{S}$ defined in \eqref{eq:S_preconditioners:scaled} is a minimiser of \eqref{eq:whSS}. Furthermore, the first $\rank(B) - r$ eigenvalues
of $\scaled{S} \inv S$ are given by $1 + \lambda_{r+i}(G)$, $i=1,\ldots, \rank(B) - r$, and the remaining
ones are given by $1$ with multiplicity $n+r-\rank(B)$.
\end{theorem}
\begin{proof}
We derive a lower bound for $\mathcal{P}\mapsto \BregmanLogDet(\mathcal{P},S)$.
Let $\mathcal{P} = V\Sigma V^*$, $\Sigma =\diag(\sigma)$ and 
$S = U \Lambda U^*$, $\Lambda = \diag(\lambda)$.
The first term of the divergence is $\trace{\mathcal{P} S\inv}$,
which can be written as
\begin{align*}
\trace{\mathcal{P} S\inv} & = \trace{V\Sigma V^* U\Lambda\inv U^*} = \trace{\Sigma P \Lambda\inv P^*},
\end{align*}
where $P = V^* U$ is unitary. Since $P\mapsto \trace{\Sigma P 
\Lambda\inv P^*}$ is
a continuous map over a compact set (the orthogonal group), it will attain its extrema in this set by
Weierstrass' theorem. As a result, we obtain the following lower bound \cite{bushell1990trace}:
\begin{align*}
\trace{\mathcal{P} S\inv} \geq \sum_{i=1}^n \sigma_i \lambda_i\inv.
\end{align*}
Now note that $I + G_r$ and $I + Q\inv B Q\invHermitian$ are simultaneously
diagonalisable so they share an eigenbasis with eigenvectors $u_i$, $i=1,\ldots n$.
Let $1 + \nu_i$ be the eigenvalues of $I + G_r$
and $1 + \mu_i$ the eigenvalues of $I + Q\inv B Q\invHermitian$.
We have, by construction, $\nu_i = \mu_i$ for $1 \leq i \leq r$, $\nu_i = 0$
for $r+1 \leq i$. Then we have the lower bound on the log determinant term as
a function of $\Sigma$, which is realised by the choice $X=G_r$ in \eqref{eq:P_choice}:
\begin{align*}
- \logdet{\mathcal{P} S\inv}
 = - \sum_{i=1}^n \log\left(\frac{1+\sigma_i}{1+\mu_i}\right)
 \geq - \sum_{i=r+1}^n \log\left(\frac{1}{1+\mu_i}\right).
\end{align*}
This implies that for any $\mathcal{P}$, we have
\begin{equation}\label{eq:Bregman:lower_bound}
\BregmanLogDet(\mathcal{P} ,S) \geq \sum_{i=1}^n \sigma_i \lambda_i\inv - \sum_{i=r+1}^n \log\left(\frac{1}{1+\mu_i}\right) - n.
\end{equation}
Using Corollary \ref{cor:divergence_scalar_eigs} and orthogonality of the
eigenvectors we obtain
\begin{align}
\BregmanLogDet(\scaled{S}, S) & = \BregmanLogDet(I + G_r, I + Q\inv B Q\invHermitian)\nonumber \\
& = \sum_{i=1}^n \sum_{j=1}^n (u_i\Hermitian u_j)^2 \left(\frac{1 + \nu_i}{1 + \mu_j} - \log\left(\frac{1 + \nu_i}{1 + \mu_j}\right) - 1\right)\nonumber \\
& = \sum_{j=r+1}^{\rank(B)} \left(\frac{1}{1 + \mu_j} - \log\left(\frac{1}{1 + \mu_j}\right) - 1\right).\label{eq:Bregman:whS}
\end{align}
This coincides with the lower bound \eqref{eq:Bregman:lower_bound}, so the first
result follows. The statement concerning the generalised eigenvalues follows since
$G_r$ and $G$ are simultaneously diagonalisable by construction.
\end{proof}

We can also ask if $G_r$ satisfies a similar result. Since the domain of the divergence
is $\Snpp\times\Snpp$ we cannot say that $G_r$ is the matrix such that
$\BregmanLogDet(G_r, G)$ is minimised. We can, however, state a
different result.
\begin{proposition}
Let $VM V\Hermitian$ be an eigendecomposition
of $G=Q\inv B Q\invHermitian$ and $V_r$ the first $r<n$ columns of $V$.
The rank $r$ approximation of $G$ obtained by
a truncated singular value decomposition, $G_r$, is a minimiser 
of
\[
\Snp \ni X \mapsto \Divergence(V_r\Hermitian X V_r, V_r\Hermitian G V_r),
\]
where $X$ is such that $V_r\Hermitian X V_r \in \Srpp$ and $\rank(X)\leq r$ and 
$\Divergence$ can be either the von Neumann or log determinant divergence.
\end{proposition}
\begin{proof}
This follows trivially from direct computation.
Indeed, $V_r\Hermitian G V_r = M_r$, where $M_r$
is the $\Rrr$ submatrix of $M$ containing the first $r$ eigenvalues of 
$G$. So, by construction,
\begin{align*}
\Divergence(V_r\Hermitian G_r V_r, V_r\Hermitian G V_r)
& = \Divergence(V_r\Hermitian V_r M_r V_r\Hermitian V_r, V_r\Hermitian VM V\Hermitian V_r)\\
& = \Divergence(M_r, M_r)\\
& = 0.
\end{align*}
Since $0$ is a lower bound of the convex function 
$X\mapsto \Divergence(V_r\Hermitian X V_r, V_r\Hermitian G V_r)$,
we are done.
\end{proof}
\cite{kulis2009low} extends the divergences to low-rank matrices, which will
be explored in more detail in the context of Nystr\"om approximations in Section
\ref{sec:matrix_nearness_problems}.

\begin{remark}\label{remark:divergence_values}
The quantities $\BregmanLogDet(\scaled{S}, S)$ and $\BregmanLogDet(\nonscaled{S}, S)$ are
described in terms of the eigenvalues of $I + G$,  $I + G_r$ and $I + Q\inv B_r
Q\invHermitian$. By construction,
\[
\lambda_i(I + G_r) = \lambda_i(I + G)
\quad \textnormal{when}\quad 1 \leq i \leq r,
\]
so the divergence $\BregmanLogDet(\scaled{S}, S)$ in
\eqref{eq:Bregman:whS} measures a quantity in terms of the
$i$th eigenvalues, where $r+1\leq i\leq n$.\\
An attempt to carry out the same analysis for $\BregmanLogDet(\nonscaled{S}, S)$
makes it clear why we expect $\scaled{S}$ to be a better preconditioner.
The derivation of the bound \eqref{eq:Bregman:whS} relies on the
fact that $I + G_r$ and $I + G$ are simultaneously
diagonalisable. In general, however, $Q\inv B_r 
Q\invHermitian \neq G_r$ (see Corollary \ref{cor:scaled_equals_nonscaled}).
Letting $(u_i, 1+\mu_i)$ and $(\tilde u_i, 1+ \tilde\nu_i)$
denote the eigenpairs 
of $I+G$ and $I+Q\inv B_r Q\invHermitian$, respectively, we have:
\begin{equation}\label{eq:Bregman_tildeS}
\BregmanLogDet(\nonscaled{S}, S) = \sum_{i=1}^n \sum_{j=1}^n (\tilde u_i\Hermitian u_j)^2\left(\frac{1+\tilde\nu_i}{1+\mu_j} - \log\left(\frac{1+\tilde\nu_i}{1+\mu_j}\right) - 1\right).
\end{equation}
The importance of the choice of basis is also apparent from the
term $(\tilde u_i\Hermitian u_j)^2$ in \eqref{eq:Bregman_tildeS}, since
this measures how aligned the bases are between $S$ and its preconditioner.
To summarise, a reason for $\scaled{S}$ appearing to be a better
preconditioner is that it matches the eigenvalues of $S$
in the first parts of the spectrum i.e. 
$[1, r]$, which directly implies unit eigenvalues
of the multiplicity mentioned in Theorem \ref{thm:whSS_minimiser}.
\end{remark}

\begin{corollary}\label{corr:BregGeq}
$\BregmanLogDet(\nonscaled{S}, S) \geq \BregmanLogDet(\scaled{S}, S)$.
\end{corollary}

While corollary \ref{corr:BregGeq} is trivial, its significance in
terms of interpreting preconditioning strategies for $S$ is not yet
clear. If $D(X,Z) < D(Y,Z)$, then we in general expect the matrix
$XZ\inv$ to be closer to the identity than $YZ\inv$, but the quality
of a preconditioner depends not only on the bounds for the eigenvalues
but also on the distribution of the spectrum. Based on Remark 
\ref{remark:divergence_values} and Theorem \ref{thm:whSS_minimiser},
the basis for the preconditioner also plays a part. This is discussed in
Section \ref{sec:iterative_solution_methods} in the context of iterative
methods. We conclude this section with a result regarding the condition
number of the preconditioned matrix:

\begin{theorem}\label{thm:kappa2_minimisation}
When $\rank(G)<n$, $G_r$ is a minimiser of the problem
\begin{subequations}\label{eq:kappa2_minimisation}
\begin{align}
\minimise_{X\in\Snp}\quad & \kappa_2(P\invhalf S P\invhalf)\\
\subto\quad
& P = Q(I+X)Q\Hermitian\label{eq:kappa2_minimisation:P}\\
& \rank(X) \leq r.
\end{align}
\end{subequations}
\end{theorem}
\begin{proof}
By definition, we have $\kappa_2(P\invhalf S P\invhalf) = \frac{\lambda_1(P\invhalf S P\invhalf)}{\lambda_n(P\invhalf S P\invhalf)}$.
We find a lower bound for $\lambda_1({P\invhalf S P\invhalf})$ independent of the choice of $X$. First, note that
\[
\lambda_1({P\invhalf S P\invhalf}) = \lambda_1(P\inv S) = \lambda_1({(I + Q\inv X Q\invHermitian)\inv (I + G))}.
\]
Then,
\[
\lambda_1({(I + Q\inv X Q\invHermitian)\inv (I + G))}
= \max_{0 \neq x\in\Cn} \frac{x\Hermitian(I+G)x}{x\Hermitian(I+ Q\inv X Q\invHermitian)x}
\geq 1 + \lambda_{r+1}(G).
\]
We now provide an upper bound on $\lambda_n(P\invhalf S P\invhalf)$. By definition,
\[
\lambda_n(P\invhalf S P\invhalf) = \lambda_n((I + Q\inv X Q\invHermitian)\inv (I + G)).
\]
We assume $\ker(G)\neq\emptyset$, and since $Q\inv X Q\invHermitian$ is positive 
semidefinite for any $X\in\Snp$ with rank at most $r$, we have
\[
\lambda_n(P\invhalf S P\invhalf) \leq 1.
\]
Consequently, we have
\[
\kappa_2(P\invhalf S P\invhalf) \geq 1 + \lambda_{r+1}(G).
\]
The choice $X = G_r$ in \eqref{eq:kappa2_minimisation:P} leads to
\[
P = Q(I + G_r)Q\Hermitian = \scaled{S}.
\]
We have $\lambda_1(\scaled{S}\inv S) = 1 + \lambda_{r+1}(G)$
thanks to Theorem \ref{thm:whSS_minimiser} so we are done.
\end{proof}

When $G$ has full rank, Theorem \ref{thm:whSS_minimiser} still holds, but
Theorem \ref{thm:kappa2_minimisation} does not apply. To see why, observe that
\[
\lambda_n(P\inv S) = \lambda_n((I + Q\inv X Q\invHermitian)\inv (I + G)) \leq \lambda_n(I + G),
\]
so the general bound for the preconditioned matrix is
\[
\kappa_2 (P\inv S) \leq \frac{1+\lambda_{r+1}(G)}{1+\lambda_n(G)}.
\]
We have $\lambda_n(\scaled{S}\inv S) = 1 \leq 1 + \lambda_n(G)$. Therefore,
$1+\lambda_{r+1}(G) = \kappa_2(\scaled{S}\inv S) \geq 
\frac{1+\lambda_{r+1}(G)}{1+\lambda_n(G)}$. However, we can still
construct a preconditioner that minimises the condition number of the 
preconditioned matrix
\begin{equation}\label{eq:fullrankG}
Q (I + G_r + \alpha (I-VV\Hermitian) ) Q\Hermitian,
\end{equation}
where the columns of $V\in\Cnr$ are the $r$ left dominant singular vectors of $G_r$
and selecting $\lambda_n(G) \leq \alpha \leq \lambda_{r+1}(G)$
since it lifts the eigenvalues of $\scaled{S}\inv S$ that are not unity to the
interval $[\lambda_n(G), \lambda_{r+1}(G)]$. A similar approach is used in
\cite{zhao2022nysadmm}. The application of \eqref{eq:fullrankG} can be applied 
at similar computational cost to \eqref{eq:S_preconditioners:scaled} but assumes knowledge
of the interval in which $\alpha$ must be chosen. In addition, the preconditioner
\eqref{eq:fullrankG} will only increase the $1$-multiplicity of the preconditioned
matrix by one. As we shall see in Section \ref{sec:iterative_solution_methods}, this
means that using \eqref{eq:fullrankG} as a preconditioner only reduces the number of 
PCG iteration required for convergence (in exact arithmetic) by one when compared to \eqref{eq:S_preconditioners:scaled}.\\

Finally, we comment on the construction and application of the preconditioner 
$\scaled{S}$, which requires two applications of $Q\inv$ and the
inversion of $(I + G_r)$ onto a vector.
We denote by $\MatVec_X$ the cost of computing matrix-vector products
with $X$. To derive and apply the proposed preconditioner we make the assumption that 
the factorisation of $A$ into $QQ\Hermitian$ is practical to compute,
and that $\MatVec_Q$ and $\MatVec_{Q\inv}$ are cheap. We often have additional 
structure so that $\MatVec_Q$ and $\MatVec_{Q\inv}$ are linear in $n$. For instance,
when $A$ is a banded matrix with bandwidth $\omega$, it can be factorised in 
$O(\omega^2 n)$ into triangular matrices with the same bandwidth, so $\MatVec_Q = 
\MatVec_{Q\inv} = O(\omega n)$. Now, assume $G_r = FF\Hermitian$ for some $F\in\Cnr$.
By the Woodbury identity,
\begin{equation}\label{eq:woodbury}
(I + FF\Hermitian)\inv =  I - F(I_r + F\Hermitian F)\inv F\Hermitian,
\end{equation}  
so given $(I_r + F\Hermitian F)\inv$, which can be constructed in $O(nr^2)$,
we can make products with $(I + FF\Hermitian)\inv$ in $O(nr)$ so the action
\[
x \mapsto \scaled{S}x = Q\invHermitian (I + FF\Hermitian)\inv Q\inv x
\]
can be applied in $O(\MatVec_Q + nr)$. Alternatively, $(I + FF\Hermitian)\inv$
can be applied given a thin QR decomposition of $F$ at a similar cost.
In general, $G=Q\inv B Q\invHermitian$ is a dense $n\times n$ matrix, so 
construction $G_r$ is impractical for large applications. In Section
\ref{sec:rsvd} we describe an alternative to the truncated SVD that constructs
an approximation of $(I + G_r)\inv$ 
in a way that avoids forming the matrix $G$.\\

In summary, we have presented the Bregman divergence framework for analysing
preconditioning of \eqref{eq:Sxb} and some of the important properties of
the Bregman divergence, e.g., Proposition \ref{prop:bregman_invariance}. 
Theorem \ref{eq:whSS}
showed that \eqref{eq:S_preconditioners:scaled} is a minimiser of a
constrained divergence and that in some cases it is a minimiser of the
condition number of the preconditioned matrix (Theorem \ref{thm:kappa2_minimisation}).
Section \ref{sec:splitting} explores extensions of this framework and Section 
\ref{sec:iterative_solution_methods} discusses the results above in the
context of classic results for iterative methods.

\subsection{Choice of Splitting}\label{sec:splitting}

So far we assumed \emph{a priori} that $S$ admits a splitting into a
positive definite term $A$ and a remainder term $B$, in which case 
Theorem \ref{thm:whSS_minimiser} tells us how to construct a preconditioner
specific of the form $A + X$, for some low-rank matrix $X$. We now
address the following  problem, where we denote by $\mathcal{C} \subset \Snpp$ 
some set of admissible factorisable matrices:
\begin{subequations}\label{eq:whSS_general}
\begin{align}
\minimise_{\substack{V\in\Snpp\\ W\in \Sn}} \quad & \BregmanLogDet(X, S)\label{eq:whSS_general:obj}\\
\subto\quad & X = V + W\label{eq:whSS_general:X}\\
&\rank(W) \leq r,\\
& V \in \mathcal{C}.\label{eq:whSS_general:Omega}
\end{align}
\end{subequations}

Without the constraint \eqref{eq:whSS_general:Omega}, an optimal solution
to \eqref{eq:whSS_general} is clearly $V=S$, $W=0$. The admissibility set
$\mathcal{C}$ can therefore be viewed as the possible choices of matrices $V$
whose factors $Q$ define the congruence transformation used in \eqref{eq:whSS}.
If it is known that $S=A+B$, $A\in\Snpp$ and $B\in\Snp$, then the solution
given in Section \ref{sec:precond_bregman} is recovered by setting the admissibility set to $\mathcal{C} =\{ A\}$.
There is no \emph{a priori} way of choosing an optimal split of an arbitrary
$S\in\Snpp$ into the sum of $A\in\Snpp$ and $B\in\Snp$ without deeper 
insight into its structure.
It is, however, useful to reason about \eqref{eq:whSS_general} to develop
heuristics for practical use. We provide an example of $\mathcal{C}$ in Section
\ref{sec:diag_pcs} that recovers the standard Jacobi preconditioner.\\

Some splits of $S$ may be more practical than others.
For instance, if $S$ is provided as a sum of a sparse indefinite term
$C_\mathrm{sparse}$ and a low-rank term $K$ (whose sum is positive definite),
then one option is to let $A= C_\mathrm{sparse} + \alpha I$, $B = K - 
\alpha I$ for some suitable constant $\alpha$. Further, if a factorisation 
of the positive definite term is not readily available, 
then an incomplete one can be used to a similar effect by compensating 
for the remainder in the low-rank approximation. 
It can even be done \emph{on-the-fly}: suppose $S = \tilde A + \tilde B$ 
where the matrix $\tilde A$ is a given candidate for factorisation (not necessarily known
to be positive definite). We attempt to perform a Cholesky factorisation,
and if a negative pivot is encountered, then it is replaced by some small positive
value. This amounts to producing a factorisation $LL\Hermitian = \tilde A + D$
for some suitable diagonal $D\in\Snp$. Then, taking $A = LL\Hermitian$ results in
the following choice of splitting:
\[
S = A + B = LL\Hermitian + (\tilde B - D).
\]
As another example, suppose we are given a matrix $S$  is "almost"
a positive definite circulant matrix, i.e., $A\in\Snpp$ is a circulant matrix
and $A \approx S$. Factorisations of circulant matrices are typically 
practical to compute, so we can construct a preconditioner from the splitting 
$S = A + B$, where $B = S - A$ can be thought of as an approximation error. 
As a result, it may be natural to fix $V=A$ in \eqref{eq:whSS_general:X} and
only minimise the objective over $W$. As a result, we let $W\in\Sn$ in
\eqref{eq:whSS_general:obj} since an error term $B =S-A$ is in general indefinite.
This approach extends
to any situation where the positive definite part of $S$ is a small perturbation
away from a structured matrix that is easier to factorise. We emphasise that this is 
highly dependent on the problem at hand and may perform arbitrarily poorly in
practice. An investigation of such approaches is deferred to future work.

\subsection{Diagonal Preconditioners}\label{sec:diag_pcs}

Another preconditioner can be obtained as a minimiser of a Bregman divergence 
by introducing different constraints to those in \eqref{eq:P_choice} and 
\eqref{eq:rank_of_X}. For instance, we can recover the standard block Jacobi 
preconditioner in the following way. Suppose we want to construct a 
preconditioner $\mathcal{P}\in\Snpp$
for the system \eqref{eq:Sxb} of the form
\begin{equation}\label{eq:blkdiag_form}
\mathcal{P} = \blkdiag(D_1,\ldots, D_b),
\end{equation}
with $b$ blocks of size $n_i$, $i=1,\ldots, b$, such that $n = \sum_{i=1}^b n_i$.
Let us denote by $E_i\in\mathbb{R}^{n \times n_i}$ the matrix constructed
from the identity matrix of order $n$ by deleting columns such that
\[
D_i = E_i\Hermitian \mathcal{P} E_i,\quad i=1,\ldots, b.
\]
Letting $\mathbf{P}$ denote matrices of the form \eqref{eq:blkdiag_form}, we formulate the following
optimisation problem:
\begin{align}\label{min:Bregman:block_diag}
\minimise_{\mathcal{P}\in\mathbf{P}} \quad &\BregmanLogDet (S, \mathcal{P}).
\end{align}
Recalling $X:Y :=\trace{X\Hermitian Y}$ and using $\BregmanLogDet (S, \mathcal{P}) = 
\BregmanLogDet (\mathcal{P}\inv, S\inv)$, the optimality conditions are:
\begin{align*}
S: E_i D_i\inv \diff D_i D_i\inv E_i\Hermitian  = \mathcal{P}: E_i D_i\inv \diff D_i D_i\inv E_i\Hermitian, \quad \forall \diff D_i\in \mathbb{H}^{n_i}_{++}, \quad i=1,\ldots,b.
\end{align*}
so $E_i\Hermitian S E_i  = D_i$. A minimiser of \eqref{min:Bregman:block_diag} is 
therefore the following block Jacobi preconditioner:
\[
\mathcal{P}^\star = \blkdiag(E_1\Hermitian S E_1,\ldots, E_b\Hermitian S E_b).
\]
Note that \eqref{min:Bregman:block_diag} applies to arbitrary $S\in\Snpp$ and
departs from the main preconditioners in \eqref{eq:S_preconditioners} 
(i.e. $A$ plus a positive semidefinite low-rank approximation). Investigating 
whether a more general class of preconditioners can be sought as minimisers to
Bregman divergences is subject to future work.

\subsection{Preconditioned Iterative Methods}\label{sec:iterative_solution_methods}

The following convergence results for PCG are well-known \cite{greenbaum1997iterative}:
\begin{theorem}
Let $S\in\Snpp$ and $x_0\in\mathbb{C}^n$ be an initial guess for the solution of the system $Sx=b$. 
Further, let $\mathcal{M}$ denote the inverse of a preconditioner. The
$k$\textsuperscript{th} iterate of PCG, denoted $x_k$, satisfies:
\begin{align*}
\| x - x_k \|_S  & = \inf_{\substack{p_k \in \mathcal{P}_k\\ p_k(0) = 1}} \| p_k(\mathcal{M}S)(x - x_0)\|_S,\\
\| x - x_k \|_S  & \leq \inf_{\substack{p_k \in \mathcal{P}_k\\ p_k(0) = 1}} \sup_{z \in \lambda(\mathcal{M}S)}|p_k(z)|\|x - x_0\|_S,\\
\| x - x_k \|_S  & \leq 2\Big(\frac{\sqrt{\kappa_2(\mathcal{M}S)} -1}{\sqrt{\kappa_2(\mathcal{M}S)} +1}\Big)^k\|x - x_0\|_S,
\end{align*}
where $\mathcal{P}_k$ is the space of polynomials of order at most $k$.
\end{theorem}
The condition number of $\mathcal{M}S$ plays a role in the convergence theory above, but the
clustering of the eigenvalues is also a factor. Recall the following theorem:
\begin{theorem}[{\cite[Theorem 10.2.5]{golub2013matrix}}]\label{thm:I+B_rankr}
If $M = I + Z\in\Snpp$ and $r = \rank (Z)$, then the conjugate gradient method converges
in at most $r+1$ iterations.
\end{theorem}

Theorem \ref{thm:whSS_minimiser} guarantees multiplicity $n-\rank(B) +r$ of unit
eigenvalues and the explicit characterisation of the remaining ones which together
with Theorem \ref{thm:I+B_rankr} leads to the following result.
\begin{corollary}\label{cor:pcg_scaled}
Let $1 + \mu_j$, $j=1,\ldots, \rank(B) - r + 1$ denote the $\rank(B) - r + 1$ distinct 
eigenvalues of $\scaled{S}\inv S$. Using $\scaled{S}$ defined in 
\eqref{eq:S_preconditioners:scaled} as a preconditioner for the system in \eqref{eq:Sxb}
for some initial guess $x_0$, the $k$\textsuperscript{th} iterate of PCG, $x_k$,
satisfies:
\begin{subequations}
\begin{align}
\| x - x_k \|_S  & \leq \inf_{\substack{p_k \in \mathcal{P}_k\\ p_k(0) = 1}} \sup_{1 \leq j \leq \rank(B) - r + 1}|p_k(1 + \mu_j)|\|x - x_0 \|_S,\label{eq:pcg:bound1}\\
\| x - x_k \|_S  & \leq 2\Big(\frac{\sqrt{\frac{1 + \lambda_{r+1}(G)}{1 + \lambda_n(G)}} -1}{\sqrt{\frac{1 + \lambda_{r+1}(G)}{1 + \lambda_n(G)}} +1}\Big)^k\|x - x_0 \|_S.\label{eq:pcg:bound2}
\end{align}
\end{subequations}
Further, assuming exact arithmetic, PCG using $\scaled{S}$ as a preconditioner for 
solving $Sx=b$ converges in at most $\rank(B) - r+1$ steps.\\

When $\rank(B)<n$, \eqref{eq:pcg:bound2} simplifies to
\[
\| x - x_k \|_S \leq 2\Big(\frac{\sqrt{1 + \lambda_{r+1}(G)} -1}{\sqrt{1 + \lambda_{r+1}(G)} +1}\Big)^k\|x - x_0 \|_S.
\]

\end{corollary}
An attractive property of $\scaled{S}$ is that it clusters the eigenvalues 
of $\scaled{S}\inv S$ leading to the supremum in \eqref{eq:pcg:bound1} being
taken over the set $\lambda_{j+r}(G)$, $j=1,\ldots, \rank(B)- r$ and zero.
If we were to use $\nonscaled{S}$,
we cannot deduce any information about the eigenvalues of $\nonscaled{S}\inv S$
or any clustering. In view of \eqref{eq:Bregman_tildeS} and Remark 
\ref{remark:divergence_values}, this preconditioner most likely leads to a 
lower multiplicity of unit eigenvalues. The bound \eqref{eq:pcg:bound2}
also shows that we expect PCG using $\scaled{S}$ to converge quickly when
$\lambda_{r+1}(G)$ is close to $\lambda_n(G)$, e.g. when the spectrum of
$G$ decays rapidly to $\lambda_n(G)$. We support the results and the intuition above with a
numerical study in Section \ref{sec:numerical_results}.

\section{Low-rank Approximations as Minimisers}\label{sec:matrix_nearness_problems}

In the previous section we posed the problem of finding a preconditioner
as a constrained optimisation problem. In this section we build on this
approach and show that several low-rank approximations, 
including \eqref{eq:S_preconditioners:scaled}, can be interpreted
as minimisers of various functionals or partial factorisations.\\

In what precedes we found $\scaled{S}$ as a minimiser of a Bregman divergence.
Below, we also find that $\scaled{S}$ is a solution
of the following \emph{scaled} norm minimisation problem, where $\|\cdot\|$ can be
either the spectral or Frobenius norm.
\begin{proposition}\label{prop:scaled:frobenius}
$\scaled{S}$ defined in \eqref{eq:S_preconditioners} is a minimiser of
\begin{subequations}
\begin{align}
\minimise_{\mathcal{P}\in\Snpp}\quad & \| Q\inv (S - \mathcal{P}) Q\invHermitian \|^2\\
\subto\quad
& \mathcal{P} = Q(I + X)Q\Hermitian\\
& \rank(X) \leq r.
\end{align}
\end{subequations}
\end{proposition}
\begin{proof}
This follows from computation:
\begin{align*}
\|  Q\inv (S - \mathcal{P}) Q\invHermitian \|_2^2
& = \|  Q\inv B  Q\invHermitian - X \|_2^2.
\end{align*}
The truncated SVD is optimal in the spectral and Frobenius norm,
so we must have $X = G_r$ and the result follows.
\end{proof}

We now adopt a more general setting to extend the result above. In light of the singular
value decomposition, the task of finding a preconditioner for a system
$Hx = b$, $H\in\Snpp$ can be thought of as matrix-nearness problem consisting of
\begin{enumerate}
    \item \emph{finding an approximation of a dominant subspace\footnote{If $H=U\Sigma V\Hermitian$ is a SVD, then the $r$\textsuperscript{th} dominant left subspace of $H$ is given by the span of the first $r$ columns of $U$.} of the range of $H$ (and in the Hermitian case also the row space)},
    \item \emph{finding an approximation of the eigenvalues of $H$},
\end{enumerate}
subject to a measure of nearness between the approximant and the matrix $H\in\Cnn$.
A truncated SVD provides an \emph{optimal truncated basis}
for the approximation. For large matrices, this is a prohibitively
costly operation. \emph{Range finders} can be used when the basis of a matrix approximation is 
unknown. These are matrices $\Theta\in\Cnr$ with orthonormal columns such that
\begin{equation}\label{eq:range_finder_epsilon}
\| H - \Theta\Theta\Hermitian H\|_F \leq \epsilon,
\end{equation}
for some threshold $\epsilon$ (step 1 above). $\Theta$ can be computed from an QR
decomposition of $H\Omega$, $\Omega\in\Rnr$, where $\Omega$ is a test or \emph{sketching}
matrix. For Hermitian matrices, this results in the following randomised approximation:
\begin{equation}\label{eq:range_approximated_matrix}
\randsvd{H} = \Theta\Theta\Hermitian H \Theta\Theta\Hermitian.
\end{equation}
In practice, this is the province of randomised linear algebra where $\Omega$ is often
drawn from a Gaussian distribution, the analysis and discussion of which is deferred to 
Section \ref{sec:rsvd}. The purpose of the remainder of this section is to discuss
the truncated singular value decomposition and the Nystr\"om approximation before 
introducing randomness. In this section, we assume $\Omega$ is simply chosen such that
\eqref{eq:range_finder_epsilon} holds.\\

The next result presents the Nystr\"om approximation \cite{drineas2005nystrom,halko2011finding}.
\begin{proposition}[Nystr\"om approximation: Frobenius minimiser]\label{prop:frob_minimiser}
Let $\Omega\in\Rnr$ be a test matrix of full rank and let
\[
\Nystrom{H} = (H\Omega) (\Omega\Hermitian H \Omega)^+ (H\Omega)\Hermitian
\]
be the Nystr\"om approximation of $H\in\Snp$ (whose Hermitian square root is denoted by
$H^\half$). Then $X^\star=(\Omega\Hermitian H \Omega)^+$ is found as the minimiser of 
the following problem where $Y=H\Omega$:
\begin{equation}\label{eq:Nystrom_functional}
\minimise_{X\in\mathbb{H}^r}\; \| (H^\half)^+ ( H - YXY\Hermitian) (H^\half)^+\|_F^2.
\end{equation}
\end{proposition}
\begin{proof}
The optimality conditions of \eqref{eq:Nystrom_functional} are:
\begin{align*}
Y^* H^+ ( H - YXY^*) H^+ Y : \diff X = 0,\quad \forall \diff X\in\Sr.
\end{align*}
Simplifying yields
\[
\Omega\Hermitian H \Omega = \Omega\Hermitian (H \Omega) X (H \Omega)\Hermitian\Omega,
\]
so $X=(\Omega\Hermitian H \Omega)^+$.
\end{proof}
It is sometimes beneficial to select $\Omega\in\mathbb{R}^{n\times (r+p)}$ for some
small integer $p$. In the context of randomised approaches, this is called
\emph{oversampling} and will be described in more detail in Section \ref{sec:rsvd}.
We briefly mention how to incorporate this into the formulation above. When
$\Omega$ is a $n\times (r+p)$ matrix, $H\Omega$ has the same dimensions. The idea
is to only use the $r$ leading left singular vectors of $H\Omega$, denoted by
$U_r$, to compute with a $n\times r$ matrix. Then,
substituting $U_r$ for $Y$ in \eqref{eq:Nystrom_functional} leads
to the solution $X^+ = U_r\Hermitian H U_r$, so
\[
H \approx U_r (U_r\Hermitian H U_r)^+ U_r\Hermitian.
\]

We can also write the Nystr\"om approximation as the minimiser of a Bregman
divergence minimisation problem. In \cite{kulis2009low}, the divergence is extended
to rank $r\leq n$ matrices $X$, $Y\in\Snp$ via the definition
\[
\BregmanLogDet^Z(X, Y) = \BregmanLogDet(Z\Hermitian X Z, Z\Hermitian Y Z),
\]
where $Z\in \Cnr$, $\rank(Z) = r$. Since any such $Z$ can be written as a product
$Z=O W$, where $O\in\Cnr$ has orthonormal columns and $W\in\Crr$ has full rank, we know that
\[
\BregmanLogDet(Z\Hermitian X Z, Z\Hermitian Y Z) = \BregmanLogDet(O\Hermitian X O, O\Hermitian Y O).
\]
This leads to the following observation.
\begin{proposition}[Nystr\"om approximation: Bregman divergence minimiser]\label{prop:Nystrom:bregman}
Suppose the hypotheses of Proposition \ref{prop:frob_minimiser} hold, and assume $H\Omega$
has full rank. Then, $H^\text{Nys}\langle\Omega\rangle$ is a minimiser of the following
optimisation problem:
\begin{subequations}\label{min:rankdeficient}
\begin{align}
\minimise_{W\in\Snp}\quad & \BregmanLogDet(\Omega\Hermitian W\Omega, \Omega\Hermitian H\Omega)\label{min:rankdeficient:constraint:functional}\\
\subto\quad 
& \range W \subseteq \range H\Omega.\label{min:rankdeficient:constraint}
\end{align}
\end{subequations}
\end{proposition}
\begin{proof}
As a result of equation \eqref{min:rankdeficient:constraint}, $W$ has the form
\[
W = Y X Y\Hermitian,
\]
for some $X\in\Srpp$ where $Y = H\Omega$. Using this to eliminate the constraint, we can write \eqref{min:rankdeficient} as
\begin{align}\label{min:rankdeficient_unconstrained}
\minimise_{X\in\Srpp}\quad & \BregmanLogDet(\Omega\Hermitian Y X Y\Hermitian \Omega, \Omega\Hermitian H\Omega).
\end{align}
The optimality conditions of \eqref{min:rankdeficient_unconstrained} are
\[
(\Omega\Hermitian H\Omega)\inv : \Omega\Hermitian Y \diff X Y\Hermitian \Omega
- (\Omega\Hermitian Y X Y\Hermitian \Omega)\invHermitian : \Omega\Hermitian Y \diff X Y\Hermitian \Omega = 0,\quad \forall \diff X\in\Srpp
\]
which leads to
\[
(\Omega\Hermitian H\Omega)\inv = (\Omega\Hermitian Y X Y\Hermitian \Omega)\invHermitian.
\]
Therefore, $X = (\Omega\Hermitian H\Omega)\inv$, and hence
\begin{equation}\label{eq:Nystrom_W}
W = (H\Omega)(\Omega\Hermitian H\Omega)\inv (H\Omega)\Hermitian = Y (Y\Hermitian\Omega)\inv Y\Hermitian.
\end{equation}
\end{proof}
Note that when $H\Omega$ does not have full rank, we make an additional
restriction to a subspace where $\Omega\Hermitian H\Omega\in \Srpp$ so that
\eqref{min:rankdeficient_unconstrained} is finite. This is the case when
$H\in\Snp$ and the inverse in \eqref{eq:Nystrom_W} becomes
a pseudoinverse. As a consequence of Proposition \ref{prop:Nystrom:bregman},
the Nystr\"om approximation has the following invariance.
\begin{corollary}
Suppose $\Nystrom{H}$ is given as in Proposition
\ref{prop:Nystrom:bregman} and let $\Theta R = \Omega$ denote a QR decomposition
of $\Omega$. Then,
\[
\Nystrom{H} = H^\text{Nys}\langle\Theta\rangle.
\]
\end{corollary}

\begin{remark}[Nystr\"om approximation: partial $LDL\Hermitian$ factorisation]\label{remark:partial_LDL}
We can also characterise the Nystr\"om approximation as a partial $LDL\Hermitian$ 
factorisation when $H\in\Snp$. Let $U = \begin{bmatrix} U_1 & U_2 \end{bmatrix}$ be a
unitary matrix with $U_1 \in \mathbb{C}^{n \times r}$ and $U_2 \in 
\mathbb{C}^{n \times (n-r)}$.
Then
\[
\tilde H = U\Hermitian HU= \begin{bmatrix} U_1\Hermitian HU_1 & U_1\Hermitian HU_2 \\ U_2\Hermitian HU_1 & U_2\Hermitian HU_2 \end{bmatrix}
 = \begin{bmatrix} \tilde H_{11} & \tilde H_{21}\Hermitian  \\ \tilde H_{21} & \tilde H_{22} \end{bmatrix}
\]
can be factorized as
\begin{align*}
\tilde H &=
\begin{bmatrix}I & 0 \\ 
\tilde H_{21} \tilde H_{11}^+ & I \end{bmatrix}
\begin{bmatrix}\tilde H_{11} & 0 \\ 0 & \tilde H_{22} - \tilde H_{21}\tilde H_{11}^+ \tilde H_{21}\Hermitian \end{bmatrix}
\begin{bmatrix} I & \tilde H_{11}^+ \tilde H_{21}\Hermitian  \\ 0 & I \end{bmatrix}\\
&= \begin{bmatrix}\tilde H_{11} & 0 \\ 
\tilde H_{21} \tilde H_{11}^+ \tilde H_{11} & I \end{bmatrix}
\begin{bmatrix}\tilde H_{11}^+ & 0 \\ 0 & \tilde H_{22} - \tilde H_{21}\tilde H_{11}^+ \tilde H_{21}\Hermitian \end{bmatrix}
\begin{bmatrix} \tilde H_{11} & \tilde H_{11} \tilde H_{11}^+ \tilde H_{21}\Hermitian  \\ 0 & I \end{bmatrix}\\ &=
L\begin{bmatrix}
D_1 & 0\\
0 & D_2
\end{bmatrix}
L\Hermitian.
\end{align*}
Setting $D_2 = 0$ in this expression we have:
\begin{align*}
L
\begin{bmatrix}
D_1 & 0\\
0 & 0
\end{bmatrix}
L\Hermitian
= \begin{bmatrix}
\tilde H_{11} \\ \tilde H_{21} \end{bmatrix} 
\tilde H_{11}^+  
\begin{bmatrix}
\tilde H_{11} \\ \tilde H_{21} 
\end{bmatrix}\Hermitian  = L D_1 L.
\end{align*}
Using $Y = HU_1$ in the definition of the Nystr\"om approximation \eqref{eq:Nystrom_W}
we obtain:
\begin{align*}
H^\textnormal{Nys}\langle U_1\rangle &= (HU_1) (U_1\Hermitian HU_1)^+ (HU_1)\Hermitian\\
&= U L D_1 L\Hermitian U\Hermitian.
\end{align*}
This is aligned with \cite{martinsson2020randomized}[Proposition 11.1], which
explains that the Nystr\"om approximation error is measured by the Schur
complement $D_2$ that was dropped in the steps above. Indeed,
\[
U\Hermitian H U - H^\textnormal{Nys}\langle U_1\rangle = 
\begin{bmatrix}
0 & 0\\
0 & D_2
\end{bmatrix}.
\]

To summarise, Proposition \ref{prop:scaled:frobenius} showed that the preconditioner 
\eqref{eq:S_preconditioners:scaled} can be also be derived from a nearness problem
in the Frobenius norm. The Nystr\"om approximation could, in similar fashion, be
found as a minimiser of an optimisation problem in the Frobenius norm or Bregman
divergence in Proposition \ref{prop:Nystrom:bregman} and Proposition 
\ref{prop:frob_minimiser}, respectively. Finally, we presented an interpretation
of this same result as a partial $LDL\Hermitian$ factorisation in Remark
\ref{remark:partial_LDL}.

\end{remark}
\section{Practical Design Using Randomised Linear Algebra}\label{sec:rsvd}

In this section, we present a way to approximate the positive semidefinite term
$G_r$ of $\scaled{S}$ efficiently using randomised linear algebra rather than computing an SVD of $G$ and
truncating it to order $r$. We do not wish to form $G$ so we instead compute a
rank $r$ \emph{randomised} approximation $G\langle\Omega\rangle$ which we recall
from \eqref{eq:range_approximated_matrix} is defined as
\[
G\langle\Omega\rangle = \Theta\Theta\Hermitian G \Theta\Theta\Hermitian,
\]
where $\Theta$ is a suitable \emph{range finder} based on a sketching matrix $\Omega$. 
This only requires $r+p$ matrix-vector products with $G$ where $p$ is the oversampling parameter.
Since $G$ is Hermitian, this is comprised of the following steps \cite{martinsson2019randomized}:
\begin{enumerate}
\item Draw a Gaussian test
matrix $\Omega \in \mathbb{R}^{n\times (r+p)}$ where $r$ is the desired rank
and $p$ is an oversampling parameter to form $Y=G\Omega$.
\item Compute an orthonormal basis $\Theta$ of $\range Y$, e.g., using a QR decomposition of $Y$.
\item Form $C = \Theta\Hermitian G \Theta$
and compute an eigenvalue decomposition $C = \tilde U \Pi \tilde U\Hermitian$.
If a rank $r$ approximation is required, truncate accordingly.
\item Compute $U = \Theta \tilde U$, then set
\begin{equation}\label{eq:Gr_Theta}
\randsvd{G} = U \Pi U\Hermitian.
\end{equation}
\end{enumerate}
This produces a practical preconditioner
\begin{equation}\label{eq:widehatS_practical}
\scaledrand{S} = Q(I + \randsvd{G})Q\Hermitian.
\end{equation}
The steps involved above and their costs, for a general matrix $G$, are therefore as follows:
\begin{itemize}
    \item Constructing the test matrix $\Omega$: $O(n(r+p))$.
    \item Compute $Y=G\Omega$ in $O((\MatVec_B + \MatVec_{Q\inv}) (r+p))$.
    \item Approximating the range $\Theta$ means performing a QR decomposition of $Y$: $O(n (r+p)^2)$. This also dominates the cost of the eigenvalue decomposition.
\end{itemize}
The algorithm only requires matrix-vector products with $B$ and $Q\inv$, which means
the matrix $G$ is never formed explicitly. 
The cost of the construction of $\scaledrand{S}$ is the sum of the cost of 
factorising the matrix $A$ and the $O(n (r+p)^2)$ cost described above.
We do not need to form
$\scaledrand{S}$ or $\randsvd{G}$ in \eqref{eq:Gr_Theta} explicitly as we only need
the action $x \mapsto \scaledrand{S}\inv x$, from which we deduce that an application
of the $\scaledrand{S}$ is $O(\MatVec_{Q\inv} + nr)$ for a rank $r$ approximation.

\subsection{Bounds for Randomised Low-rank Approximation}

In practice, we have $\scaledrand{S} \neq \scaled{S}$, which means that
results such as Theorem  \ref{thm:whSS_minimiser} no longer apply. The following 
theorem estimates the error in the Bregman divergence:
\begin{theorem}[Expected suboptimality of $\scaledrand{S}$]\label{thm:expectation_of_sr}
Let $\scaled{S}$ be a minimiser of \eqref{eq:whSS} and $\scaledrand{S}$ be given by
\eqref{eq:widehatS_practical} where $\rank(\randsvd{G})= r+p\leq \rank(B)$ and the oversampling parameter satisfies $p\geq 2$. Then,
\begin{equation}\label{eq:expectation_of_sr:absolute}
\expect{\BregmanLogDet(\scaledrand{S}, S)} - \BregmanLogDet(\scaled{S}, S) \leq 
2 \epsilon c_r,
\end{equation}
where
\[
\epsilon = \sqrt{\left(1 + \frac{r}{p - 1}\right)\sum_{j=r+1}^n \lambda_j(G)^2},\qquad
c_r = \|(I+G)\inv\|_F  + r (1 + \lambda_{\rank(B)}(G))\inv.
\]
Furthermore, the relative error estimate holds:
\begin{align*}
&\frac{\expect{\BregmanLogDet(\scaledrand{S}, S)} - \BregmanLogDet(\scaled{S}, S)}{\BregmanLogDet(\scaled{S}, S)}\\
\leq \; &2 c_r \sqrt{\frac{(1 + \frac{r}{p - 1})}{n-r}} \frac{\lambda_{r+1}(G)}{(\frac{1}{1+ \lambda_{\rank(B)}(G)} - \log(\frac{1}{1+ \lambda_{\rank(B)}(G)}) -1)}.
\end{align*}
\end{theorem}

\begin{proof}
We start by noting that
\begin{align*}
\expect{\BregmanLogDet(\scaledrand{S}, S) - \BregmanLogDet(\scaled{S}, S)}
& = \expect{\BregmanLogDet(I + \randsvd{G}, I + G) - \BregmanLogDet(I + G_r, I + G)}\\
& = \expect{\trace{(\randsvd{G} - G_r)(I + G)\inv}}\\
& \qquad + \expect{\log\det((I + G_r) (I + G)\inv) - \log\det((I + \randsvd{G}) (I + G)\inv)}\\
& = \expect{\trace{(\randsvd{G} - G_r)(I + G)\inv}}\\
& \qquad + \expect{\log\det(I + G_r) - \log\det(I + \randsvd{G})}\\
& = T_1 + T_2.
\end{align*}
We have the following bound for $T_1$:
\begin{align*}
T_1
& \leq \|(I+G)\inv\|_F \expect{\|\randsvd{G} - G_r\|_F}\\
& \leq \|(I+G)\inv\|_F \expect{\|\randsvd{G} - G\|_F}.
\end{align*}
Next, we treat $T_2$:
\begin{align*}
T_2
& = \expect{\log\det\Big((I + G_r)(I + \randsvd{G})\inv\Big)}\\
& = \expect{\log \frac{\prod_{1\leq i\leq n}(1 + \lambda_i(G_r))}{\prod_{1\leq i\leq n}(1 + \lambda_i(\randsvd{G}))}}\\
& \leq \expect{\sum_{1\leq i\leq r} |\log(1 + \lambda_i(G_r)) - \log(1 + \lambda_i(\randsvd{G}))|}.
\end{align*}
Since $\log$ is Lipschitz on $[x, \infty]$ with Lipschitz constant $x\inv$, 
\begin{align*}
T_2 & \leq (1 + \lambda_{\rank(B)}(G))\inv \expect{\sum_{1\leq i\leq r} |\lambda_i(G_r) - \lambda_i(\randsvd{G})|}.
\end{align*}
For $X,Y\in\Snp$, we know that
\begin{equation}\label{eq:horn_ineq}
\lambda_{i+j-1}(X+Y) \leq \lambda_1(X) + \lambda_j(X), \quad 1\leq i,j\leq n, \quad i+j-1\leq n,
\end{equation}
see, e.g., \cite[Theorem 3.3.16]{horn1994topics}. Choosing $X = G - \randsvd{G}$,
$Y=\randsvd{G}$, and setting $i=1$ in \eqref{eq:horn_ineq}, we obtain
\[
|\lambda_j(G) - \lambda_j(\randsvd{G})| \leq \lambda_1(G-\randsvd{G}).
\]
Since $\lambda_1(G-\randsvd{G}) = \| G-\randsvd{G} \|_2 \leq \| G-\randsvd{G} \|_F$, we conclude that
\begin{align*}
T_2 & \leq r (1 + \lambda_{\rank(B)}(G))\inv \expect{\| G-\randsvd{G} \|_F}.
\end{align*}
Combining the results for $T_1$ and $T_2$, we get
\begin{align*}
\expect{\BregmanLogDet(\scaledrand{S}, S) - \BregmanLogDet(\scaled{S}, S)} \leq (\|(I+G)\inv\|_F  + r (1 + \lambda_{\rank(B)}(G))\inv) \expect{\|\randsvd{G} - G\|_F}.
\end{align*}
Thanks to \cite[Theorem 10.5]{halko2011finding},
\[
\expect{\| G - \Theta\Theta\Hermitian G \|_F} \leq \epsilon = \sqrt{(1 + \frac{r}{p - 1})  \sum_{j=r+1}^n \lambda_j(G)^2},
\]
so $\expect{\| G - \randsvd{G} \|_F} \leq 2\epsilon$.\\

For the relative error, recall \eqref{eq:Bregman:whS} to notice that
\[
\BregmanLogDet(\scaled{S}, S) \geq (n-r)(\frac{1}{1+ \lambda_{\rank(B)}(G)} - \log(\frac{1}{1+ \lambda_{\rank(B)}(G)}) - 1)
\]
and
\[
\epsilon \leq \sqrt{(1 + \frac{r}{p - 1})(n-r) (\lambda_{r+1}(G))^2 } = \sqrt{(1 + \frac{r}{p - 1})(n-r)} \lambda_{r+1}(G).
\]
\end{proof}

To achieve a desired accuracy $\epsilon$ such that
\[
\|G - \Theta\Theta\Hermitian G\Theta\Theta\Hermitian\|_2 \leq 2\epsilon,
\]
it is observed that a small increase $p$ in the number of columns of the test matrix
$\Omega$ is necessary depending on the structure of $G$, see
\cite[Theorem 10.6]{halko2011finding} for details. We can also obtain a deviation 
bound:
\begin{theorem}\label{thm:point_probability_of_sr}
Let the hypotheses of Theorem \ref{thm:expectation_of_sr} hold and
assume $p\geq 4$. Further, let $e$ denote Euler's number. Then, for all $u,t\geq 1$,
\[
\BregmanLogDet(\scaledrand{S}, S) - \BregmanLogDet(\scaled{S}, S) \leq 
2 c_r \left[ \left(1 + t\sqrt{\frac{3 r}{p+1}}\right)\sqrt{ \sum_{j=r+1}^n \lambda_j(G)^2} + ut\frac{e\sqrt{r + p}}{p + 1} \lambda_{r+1}(G)\right],
\]
holds with probability $1 - 2t^{-p} + e^{-u^2/2}$, where $c_r$ is
defined in Theorem \ref{thm:expectation_of_sr}.
\end{theorem}
\begin{proof}
By \cite[Theorem 10.7]{halko2011finding},
\[
\| G - \randsvd{G} \|_F \leq 2\left[(1 + t\sqrt{\frac{3 r}{p+1}})\sqrt{ \sum_{j=r+1}^n \lambda_j(G)^2} + ut\frac{e\sqrt{r + p}}{p + 1}\lambda_{r+1}(G)\right]
\]
holds with probability $1 - 2t^{-p} + e^{-u^2/2}$ for any $u,t\geq 1$.
\end{proof}

Finally, a \emph{power range} finder can also be used to improve the accuracy of
the approximation of a target matrix $G$, where $\Theta$ in \eqref{eq:range_finder_epsilon} is computed
from a QR decomposition of 
\[
Y = (GG\Hermitian)^q G\Omega,
\]
for some $q\geq 1$, leading to
\begin{equation}\label{eq:scaledrandpower}
\scaledrandpower{S} = Q(I + \randsvdpower{G})Q\Hermitian.
\end{equation}
This produces a more accurate approximation at the cost of $2q$ more products
with $G$ than its simpler ($q=0$) counterpart \cite[Algorithm 4.3]{halko2011finding}.\\

It has been observed \cite{gittens2016revisiting,halko2011finding} that
$\NystromQ{G}$, where $\Theta$ is obtained from a QR decomposition of $G\Omega$, is a considerable improvement over $\Nystrom{G}$ with a computational
cost that compares to the randomised SVD, $\randsvd{G}$. This is essentially
because we perform a power iteration:
\[
(G \Theta) (\Theta\Hermitian G \Theta)\inv (G \Theta)\Hermitian =
(G^2 \Omega) (\Omega\Hermitian G^3 \Omega)\inv (G^2 \Omega)\Hermitian.
\]
See also 
\cite{nakatsukasa2020fast} for a generalisation based on this observation.
In Section \ref{sec:numerical_results}, we therefore evaluate the preconditioners 
above as well including the scaled and nonscaled Nystr\"om preconditioner
\begin{subequations}
\begin{align}
& \scaledNysQ{S} = Q(I + \NystromQ{G})Q\Hermitian,\label{eq:scaledNys}\\
& \nonscaledNysQ{S} = A + \NystromQ{B}, \label{eq:nonscaledNys}
\end{align}
\end{subequations}
where in \eqref{eq:nonscaledNys}, $\Theta$ is obtained from a QR decomposition of 
$B\Omega$.
\subsection{Single View Approach}

We now look at a \emph{single view} approach of the randomised SVD, where
we \emph{only once} access the matrix we wish to approximate
\cite[Section 5.5]{halko2011finding}. The algorithm is similar to the one
in the preceding section, where the range is approximated via $\Theta$ such
that $\|G - \Theta\Theta\Hermitian G\Theta\Theta\Hermitian\|_2$ is below
some tolerance. Suppose $\randsvd{G}$ is of the form
\begin{equation}\label{eq:Gr_Theta:single_pass}
\randsvd{G} = \Theta \Pi \Theta\Hermitian,
\end{equation}
where $\Pi = \Theta\Hermitian G \Theta$ with no oversampling being used. Multiplying 
both sides of $\Pi = \Theta\Hermitian G \Theta$ yields a system of equations
for $\Pi$:
\begin{equation}\label{eq:single_pass_assumption}
\Pi \Theta\Hermitian \Omega = \Theta\Hermitian \underbrace{G \Theta\Theta\Hermitian}_{\approx G} \Omega \approx \Theta\Hermitian G\Omega.
\end{equation}
$\Theta\Hermitian\Omega$ is invertible with high probability, so computing
$\Theta$ and solving the system \eqref{eq:single_pass_assumption} for $\Pi$ can be
done in $O(nr^2)$. In general, a Hermitian solution
of \eqref{eq:single_pass_assumption} is sought after owing to the definition
\eqref{eq:Gr_Theta:single_pass} by using least squares \cite{martinsson2019randomized}.
As $G$ is square, a direct inversion
still results in a symmetric solution $\randsvd{G}$ in exact arithmetic,
and the effect of the approximation step in \eqref{eq:single_pass_assumption}
is  explained in the following lemma:

\begin{lemma}
When \eqref{eq:single_pass_assumption} holds so $\Pi = \Theta\Hermitian G\Omega 
(\Theta\Hermitian \Omega)\inv$, the single pass randomised SVD
\eqref{eq:Gr_Theta:single_pass} produces a Nystr\"om
approximation $\Nystrom{G}$.
\end{lemma}
\begin{proof}
This follows from direct calculation. Note that $\Theta = G\Omega R\inv$, so expanding
$\randsvd{G}$,
\begin{align*}
\randsvd{G}
& = \Theta \Pi \Theta\Hermitian\\
& = \Theta \Theta\Hermitian G \Omega (\Theta\Hermitian \Omega)\inv \Theta\Hermitian\\
& = \Theta \Theta\Hermitian G \Omega (R\invHermitian\Omega\Hermitian G\Omega)\inv \Theta\Hermitian\\
& = \Theta \Theta\Hermitian G \Omega (\Omega\Hermitian G\Omega)\inv R\Hermitian\Theta\Hermitian\\
& = G \Omega (\Omega\Hermitian G\Omega)\inv (G\Omega)\Hermitian,
\end{align*}
where the last equality holds since $\range G\Omega = \range \Theta$. This can also
be seen since $R =\Theta\Hermitian G \Omega$.
\end{proof}
\section{Numerical Experiments}\label{sec:numerical_results}

In Section \ref{sec:synthetic_data} we present results from several numerical
experiments using synthetic data to develop an understanding of the many
different preconditioners studied above. Next, Section \ref{sec:effect_low_rank} 
attempts to explain these differences by looking at the Bregman log determinant
divergence. Finally, Section \ref{sec:numerical_results:var} applies the proposed 
methodology to a problem in variational data assimilation.


\subsection{Synthetic Data}\label{sec:synthetic_data}

\subsubsection{Experimental Setup}

We construct $S = A+B$, where $A\in\Snpp$ and $B\in\Snp$ of rank $m < n$,
for use in our numerical experiments. 
We vary the spectrum of $A$ and $B$ according to
\begin{subequations}\label{eq:spectrum}
\begin{align}
& \lambda_A(i) = \exp\{-(\alpha_A i/n  - c_A)^{\beta_A}\} + \kappa,\label{eq:spectrum_A}\\
& \lambda_B(i) = \exp\{-(\alpha_B i/m - c_B)^{\beta_B}\}.\label{eq:spectrum_B}
\end{align}
\end{subequations}
for positive parameters $\alpha_A$, $\alpha_B$, $\beta_A$, $\beta_B$, $c_A$, $c_B$, and $\kappa$ controlling the decay of the eigenvalues as a function of $i$. The parameter 
$\kappa$ is there to bound the eigenvalues of $A$ away from zero when $i$ approaches $n$. 
We set $n=1000$, $m = 600$, and $r=300$ in our experiments. The parameters used for $A$ and $B$ in \eqref{eq:spectrum} are shown in Table \ref{table:AB_params} and a graphs
of the spectra are shown in Figure \ref{fig:AB_spectra}. Selecting $\Sigma_A\in\Rnn$ and 
$\Sigma_B\in\Rmm$ as above, we then construct
\begin{subequations}\label{eq:A_B_simple}
\begin{align}
& A = O_A \Sigma_A O_A\Hermitian,\\
& B  = O_B \Sigma_B O_B\Hermitian,
\end{align}
\end{subequations}
where the columns of $O_A\in \Rnn$, $O_B\in \Rnm$ are computed from QR 
decompositions of matrices whose entries are drawn from the standard 
Gaussian distribution.

\begin{table}[tbhp]
\footnotesize
    \caption{Parameters used to generate eigenvalues of $A$ and $B$. Visualisations of these spectra can be seen in Figure \ref{fig:AB_spectra}.}\label{table:AB_params}
    \centering
\begin{tabular}{|c|l|l|l|c|l|c|}\hline
\textbf{Matrix} & $\alpha$ & $c$ & $\beta$ & $\kappa$ & \textbf{Description} & \textbf{Label}\\\hline
\multirow{4}{*}{$A$} & $0$ & $0$ & $0$ & $0.70$ & flat & 1 \\
    & $3.5$ & $0$ & $1.0$ & $0.05$ & exponential & 2\\
    & $4.0$ & $0.30$ & $4.5$ & $0.05$ & drop-off and fast decay & 3\\
    & $2.0$ & $0.25$ & $4.5$ & $0.05$ & drop-off at $n/2$ & 4\\\hline
\multirow{2}{*}{$B$} & $3.0$ & $0$ & $1.0$ & - & exponential & 1\\
    & $2.5$ & $0.55$ & $4.7$ & - & slow decay & 2\\ \hline
\end{tabular}
\end{table}
\begin{figure}
    \centering
    \includegraphics[scale=0.4]{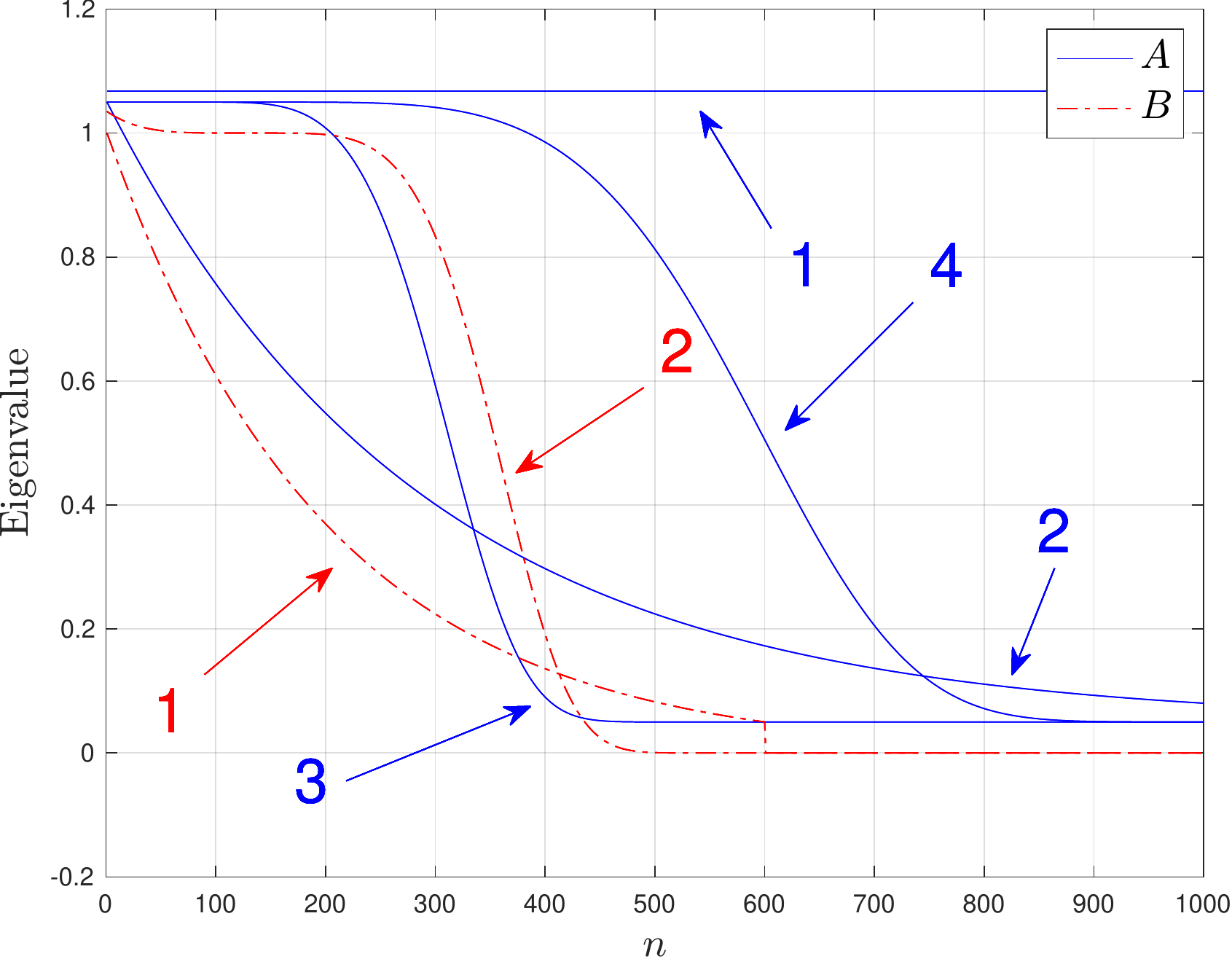}
    \caption{Different choices of spectra for $A$ and $B$ using the parameters in Table \ref{table:AB_params}.}
    \label{fig:AB_spectra}
\end{figure}

Next, we evaluate and compare the preconditioners in Table \ref{table:all_pcs}
and \ref{table:algebraic} for each
combination of $A$ and $B$. We do not use any oversampling (i.e. $p=0$) to develop
an understanding of the preconditioners in the simplest setting possible for relatively
easy problems. In our randomised power range approximations $\scaledrandpower{S}$ and 
$\nonscaledrandpower{S}$, we use $q=2$. We compare against several algebraic 
preconditioners. Since $S$ in our setup is dense we do not use an incomplete Cholesky
preconditioner. Instead, we shall use a \emph{partial} Cholesky preconditioner (see e.g.
\cite{bellavia2013matrix}). Consider a formal partition of $S$:
\[
S = \begin{bmatrix} S_{11} & S_{21}\Hermitian \\ S_{21} & S_{22}\end{bmatrix}.
\]
Let $\begin{bmatrix} F_{11}\\ F_{21}\end{bmatrix}\in\Cnr$ denote the first $r$ columns
of a Cholesky factor of $S$ using diagonal pivoting (for simplicity we ignore permutations).
We let $D_\textnormal{Schur}$ denote the diagonal of whose elements are given by the Schur 
complement
\[
S_{22} - S_{21} S_{11}\inv S_{21}\Hermitian
\]
and form the preconditioner
\[
P_\texttt{pchol} = \begin{bmatrix} F_{11} & \\ F_{21} & I_{n-r}\end{bmatrix}
\begin{bmatrix}I_r & \\  & D_\textnormal{Schur}\end{bmatrix}
\begin{bmatrix} F_{11} & F_{21}\Hermitian\\ & I_{n-r}\end{bmatrix}.
\]
Hence, this can be considered a "positive definite plus low rank" preconditioner.
Another algebraic preconditioner is the symmetric Gauss-Seidel preconditioner\footnote{Also
called the \emph{Symmetric Successive Over-Relaxation} (SGS) preconditioner with weight equal
to one.} \cite{scott2023algorithms} based on the decomposition
\[
S = L + U + D
\]
into the strictly lower triangular, upper triangular and diagonal elements:
\[
P_\textnormal{SGS} = (D + L)D\inv (D + U).
\]
This preconditioner requires access to the entire matrix, whereas $\scaled{S}$ (and
the variations thereof) require only the factors of $A$. Here we analyse both 
strategies for synthetic data, but the practicality of either strategy may be
entirely problem-dependent. The Jacobi preconditioner is simply the matrix
consisting of the diagonal entries of $S$.
The Python package \texttt{scaled\_preconditioners}\footnote{Available at 
\url{https://github.com/andreasbock/scaled_preconditioners}.} contains implementations 
of the preconditioners listed in the ``Scaled'' column of Table~\ref{table:all_pcs}.
We have used MATLAB R2022b \cite{MATLAB} to generate the results shown in this paper.

\begin{table}[tbhp]
\caption{The preconditioners based on the structure $S=A+B$.}\label{table:all_pcs}
\centering
\begin{tabular}{|l|l|l|}\hline
        \textbf{Approximation} & \textbf{Scaled} & \textbf{No scaling}\\\hline
        Truncated SVD & $\scaled{S}$ (eq. \eqref{eq:S_preconditioners:scaled}) & $\nonscaled{S}$ (eq. \eqref{eq:S_preconditioners:nonscaled})\\
        Randomised SVD & $\scaledrand{S}$ (eq. \eqref{eq:widehatS_practical}) & $\nonscaledrand{S} := A + \scaledrand{B})$\\
        Rand. power range SVD & $\scaledrandpower{S}$ (eq. \eqref{eq:scaledrandpower}) & $\nonscaledrandpower{S} := A + \scaledrandpower{B}$\\
        Nystr\"om & $\scaledNys{S}$ (eq. \eqref{eq:scaledNys}) & $\nonscaledNys{S}$ (eq. \eqref{eq:nonscaledNys}\\ \hline
    \end{tabular}
\end{table}

\begin{table}[tbhp]
\caption{The algebraic preconditioners used in our numerical experiments.}\label{table:algebraic}
\centering
        \begin{tabular}{|l|}\hline
        \textbf{Algebraic}\\\hline
        Partial Cholesky\\
        Symmetric Gauss-Seidel\\
        Jacobi\\\hline
    \end{tabular}
\end{table}

\subsubsection{Performance of Iterative Methods}

For each combination of the spectra of $A$ and $B$ depicted in Figure
\ref{fig:AB_spectra} we generate 25 random instances of $O_A$ and $O_B$ in \eqref{eq:A_B_simple} to construct the system matrix $S=A+B$.
For each of the preconditioners $\mathcal{P}$ in Table \ref{table:all_pcs} and
\ref{table:algebraic}
we study the value of the Bregman divergence $\BregmanLogDet(\mathcal{P}, S)$,
and the convergence behaviour of the PCG for the solution of $Sx=b$ with a
relative tolerance of $10^{-7}$.
Overall, we observed negligible variance in the 
results as a function of $O$ and $b$. We therefore only present results for one
instance of bases for $A$ and $B$.\\

{
\setlength{\tabcolsep}{8.15pt}
\begin{table}[h!]
\tiny
\caption{PCG iteration counts for different instances of $A$ and $B$ and choice of preconditioner. We highlight the smallest iteration count across the preconditioners
in bold.}\label{table:synthetic_data:iteration}
\begin{tabular*}{\textwidth}{@{} *{20}{c} }
\toprule 
 \multicolumn{2}{c}{Label} & \multirow{2.2}{*}{$\kappa_2(S)$} & \multicolumn{12}{c}{Iteration count}\\ \cmidrule(l{0pt}r{2pt}){1-2} \cmidrule(l{2pt}r{0pt}){4-15}
$A$ & $B$ & & $I$ & $P_\texttt{pchol}$ & $P_\textnormal{SGS}$  & Jacobi & $\nonscaled{S}$ & $\nonscaledrand{S}$ & $\nonscaledrandpower{S}$ & $\nonscaledNys{S}$ & $\scaled{S}$ & $\scaledrand{S}$ & $\scaledrandpower{S}$ & $\scaledNys{S}$ \\
\midrule
\begin{filecontents*}{results_N=1000_r=300_tex.csv}
A,B,condS,resnopc,respchol,resgaussseidel,resjacobi,resnonscaled,resnonscaledsr,resnonscaledsrq,resnonscalednys,resscaled,resscaledsr,resscaledsrq,resscalednys,iternopc,iterpchol,itergaussseidel,iterjacobi,iternonscaled,iternonscaledsr,iternonscaledsrq,iternonscalednys,iterscaled,iterscaledsr,iterscaledsrq,iterscalednys,timenopc,timepchol,timegaussseidel,timejacobi,timenonscaled,timenonscaledsr,timenonscaledsrq,timenonscalednys,timescaled,timescaledsr,timescaledsrq,timescalednys,flagnopc,flagpchol,flaggaussseidel,flagjacobi,flagnonscaled,flagnonscaledsr,flagnonscaledsrq,flagnonscalednys,flagscaled,flagscaledsr,flagscaledsrq,flagscalednys,divgnopc,divgpchol,divggaussseidel,divgjacobi,divgnonscaled,divgnonscaledsr,divgnonscaledsrq,divgnonscalednys,divgscaled,divgscaledsr,divgscaledsrq,divgscalednys
1,1,3.6e+01,8.4e-08,2.1e-08,3.4e-08,1.9e-08,4.4e-09,5.2e-08,1.7e-08,7.7e-08,4.4e-09,6.8e-08,2.6e-08,9.0e-08,9,9,\textbf{5},10,6,8,6,7,6,8,6,7,2.7e-03,7.4e-02,9.9e-02,\textbf{7.3e-03},1.1e-01,1.7e-01,9.8e-02,1.4e-01,3.5e-01,3.7e-01,3.3e-01,3.8e-01,0,0,0,0,0,0,0,0,0,0,0,0,3.3e+01,7.3e+00,\textbf{5.3e-01},1.6e+01,1.7e+00,5.1e+00,1.9e+00,5.4e+00,1.7e+00,5.1e+00,1.9e+00,5.3e+00
1,2,7.8e+01,8.2e-08,2.2e-08,8.5e-08,2.3e-08,1.5e-08,5.5e-08,5.2e-08,5.4e-08,1.5e-08,5.1e-08,5.7e-08,5.3e-08,9,10,\textbf{5},10,9,10,9,9,9,10,9,9,1.9e-03,8.9e-02,9.2e-02,\textbf{1.1e-02},1.6e-01,2.3e-01,1.5e-01,1.7e-01,4.7e-01,5.6e-01,5.1e-01,5.8e-01,0,0,0,0,0,0,0,0,0,0,0,0,7.3e+01,2.1e+01,\textbf{2.3e+00},4.5e+01,7.5e+00,1.3e+01,8.5e+00,1.4e+01,7.5e+00,1.3e+01,8.5e+00,1.4e+01
2,1,6.8e+02,7.1e-08,7.7e-08,6.8e-08,6.7e-08,6.9e-08,6.6e-08,7.6e-08,3.9e-08,6.4e-08,4.8e-08,9.0e-08,8.9e-08,31,25,13,31,10,17,11,14,\textbf{9},17,10,13,4.2e-03,2.3e-01,3.1e-01,\textbf{2.9e-02},2.8e-01,4.1e-01,2.9e-01,3.0e-01,5.4e-01,9.8e-01,6.3e-01,7.2e-01,0,0,0,0,0,0,0,0,0,0,0,0,1.3e+03,1.7e+02,1.0e+02,4.0e+02,2.7e+01,5.1e+01,2.8e+01,5.2e+01,\textbf{2.2e+01},4.5e+01,2.3e+01,4.6e+01
2,2,1.0e+03,9.9e-08,9.1e-08,7.6e-08,9.8e-08,7.8e-08,8.4e-08,4.2e-08,8.5e-08,9.8e-08,5.6e-08,4.2e-08,1.0e-07,34,29,15,34,18,24,20,20,\textbf{14},24,16,19,4.2e-03,2.6e-01,2.8e-01,\textbf{2.8e-02},3.5e-01,3.8e-01,3.6e-01,3.4e-01,7.9e-01,1.2e+00,9.0e-01,9.5e-01,0,0,0,0,0,0,0,0,0,0,0,0,1.3e+03,3.0e+02,2.0e+02,7.2e+02,5.5e+01,7.8e+01,5.8e+01,7.8e+01,\textbf{4.2e+01},7.0e+01,4.4e+01,7.0e+01
3,1,2.1e+03,8.8e-08,6.4e-08,6.7e-08,9.2e-08,4.4e-08,7.0e-08,7.9e-08,5.6e-08,3.7e-08,6.5e-08,8.0e-08,6.2e-08,48,41,22,48,15,26,16,20,\textbf{14},26,15,20,6.0e-03,3.5e-01,3.8e-01,\textbf{4.7e-02},2.7e-01,4.4e-01,3.1e-01,3.4e-01,7.1e-01,1.3e+00,7.9e-01,1.1e+00,0,0,0,0,0,0,0,0,0,0,0,0,3.4e+03,6.0e+02,5.6e+02,1.5e+03,7.4e+01,1.2e+02,7.6e+01,1.2e+02,\textbf{6.0e+01},1.0e+02,6.2e+01,1.0e+02
3,2,3.2e+03,7.4e-08,8.6e-08,8.3e-08,9.3e-08,7.9e-08,7.2e-08,6.4e-08,7.4e-08,5.1e-08,9.7e-08,6.0e-08,9.0e-08,52,47,25,52,27,38,30,31,\textbf{21},37,23,29,5.5e-03,4.1e-01,5.3e-01,\textbf{3.6e-02},5.5e-01,7.0e-01,6.4e-01,7.8e-01,1.2e+00,2.2e+00,1.4e+00,1.7e+00,0,0,0,0,0,0,0,0,0,0,0,0,3.9e+03,1.1e+03,9.4e+02,2.5e+03,1.0e+02,1.4e+02,1.1e+02,1.4e+02,\textbf{7.9e+01},1.2e+02,8.0e+01,1.2e+02
4,1,1.0e+03,8.5e-08,9.6e-08,7.8e-08,8.4e-08,2.5e-08,3.9e-08,6.2e-08,5.4e-08,3.9e-08,4.0e-08,3.0e-08,6.2e-08,40,33,17,40,12,18,12,15,\textbf{8},16,9,13,4.2e-03,2.8e-01,3.4e-01,\textbf{3.1e-02},2.2e-01,3.8e-01,2.8e-01,2.8e-01,4.9e-01,9.3e-01,4.9e-01,8.1e-01,0,0,0,0,0,0,0,0,0,0,0,0,1.0e+03,3.9e+02,1.9e+02,7.7e+02,2.2e+01,4.1e+01,2.2e+01,4.3e+01,\textbf{8.4e+00},2.4e+01,8.8e+00,2.4e+01
4,2,1.4e+03,9.6e-08,6.6e-08,5.9e-08,9.6e-08,7.1e-08,8.9e-08,6.6e-08,6.8e-08,4.2e-08,9.1e-08,2.8e-08,5.6e-08,43,38,20,43,18,24,19,21,\textbf{11},20,12,17,4.5e-03,3.3e-01,4.6e-01,\textbf{2.8e-02},3.4e-01,4.6e-01,3.7e-01,4.3e-01,7.1e-01,1.1e+00,7.6e-01,9.9e-01,0,0,0,0,0,0,0,0,0,0,0,0,9.3e+02,4.6e+02,2.4e+02,9.1e+02,4.1e+01,6.1e+01,4.4e+01,6.2e+01,\textbf{1.9e+01},4.0e+01,2.0e+01,3.9e+01
\end{filecontents*}
\csvreader[
head to column names,
]{results_N=1000_r=300_tex.csv}{}{\A & \B & \condS & \iternopc & \iterpchol & \itergaussseidel & \iterjacobi & \iternonscaled & \iternonscaledsr & \iternonscaledsrq & \iternonscalednys & \iterscaled & \iterscaledsr & \iterscaledsrq & \iterscalednys\\}
\\[-\normalbaselineskip]\hline
\end{tabular*}
\end{table}
}

{
\setlength{\tabcolsep}{1.1pt}
\begin{table}[h!]
\tiny
\caption{Bregman divergence values between the different preconditioners and choice of system matrix $S$ for different instances of $A$ and $B$. We highlight the smallest values across the preconditioners in bold.}\label{table:synthetic_data:divg}
\begin{tabular*}{\textwidth}{@{\extracolsep{\fill}} *{20}{c} }
\toprule 
 \multicolumn{2}{c}{Label} & \multirow{2.2}{*}{$\kappa_2(S)$} & \multicolumn{12}{c}{Bregman divergence $\BregmanLogDet(\cdot, S)$}\\ \cmidrule(l{0pt}r{2pt}){1-2} \cmidrule(l{2pt}r{0pt}){4-15}
$A$ & $B$ & & $I$ & $P_\texttt{pchol}$ & $P_\textnormal{SGS}$  & Jacobi & $\nonscaled{S}$ & $\nonscaledrand{S}$ & $\nonscaledrandpower{S}$ & $\nonscaledNys{S}$ & $\scaled{S}$ & $\scaledrand{S}$ & $\scaledrandpower{S}$ & $\scaledNys{S}$ \\
\midrule
\csvreader[
head to column names,
]{results_N=1000_r=300_tex.csv}{}{\A & \B & \condS & \divgnopc & \divgpchol & \divggaussseidel & \divgjacobi & \divgnonscaled & \divgnonscaledsr & \divgnonscaledsrq & \divgnonscalednys & \divgscaled & \divgscaledsr & \divgscaledsrq & \divgscalednys\\}
\\[-\normalbaselineskip]\hline
\end{tabular*}
\end{table}
}

The results are shown in Table \ref{table:synthetic_data:iteration} and
\ref{table:synthetic_data:divg}. First we comment on the algebraic preconditioners.
The simplest preconditioner, Jacobi, performs no better than unpreconditioned runs.
The symmetric Gauss-Seidel preconditioner performs well overall, and for the 
case of constant diagonal $A$, PCG requires the fewest iterations with this 
preconditioner, and is the closest to $S$ in terms of Bregman divergence.
In all other cases, $\scaled{S}$ performs better.
Overall, the partial Cholesky preconditioner reduces the number of iterations
required when compared to an unpreconditioned run, but not as effectively as the
scaled preconditioners. As mentioned, the partial Cholesky preconditioner can be
thought of as a diagonal plus low-rank matrix. 
The invariance property of the Bregman divergence was key to the theoretical
results given in Section \ref{sec:precond_bregman}.
By Theorem \ref{thm:whSS_minimiser}, $\scaled{S}$ is optimal in the Bregman 
divergence and $\scaled{S}\inv S$ has $\rank (B) -r + 1 $ distinct eigenvalues.
In general, $P_\texttt{pchol}\inv S$ will have more distinct eigenvalues. \\

Next we comment on the preconditioners in Table \ref{table:all_pcs}.
When the spectrum of $A$ is flat, Corollary \ref{cor:scaled_equals_nonscaled} 
tells us the preconditioners $\scaled{S}$ and $\nonscaled{S}$ (and their respective 
variants) are identical. This is confirmed by the results seen in the top two
rows of Table \ref{table:synthetic_data:iteration}. As expected, most of the
preconditioners perform similarly here as the problem is well-conditioned.
In this case, the randomised 
power range preconditioners are almost as good as $\scaled{S}=\nonscaled{S}$ in 
terms of PCG convergence, with the results for the Nystr\"om approximation being
slightly worse. The standard randomised SVD can be seen be the worst choice in all
cases. Analysing the Bregman divergence values in Table \ref{table:synthetic_data:divg}
leads to the same conclusion.\\

The rows of Table \ref{table:synthetic_data:iteration} and \ref{table:synthetic_data:divg} where label $A=2$ show the results for the simple exponential decay of
the eigenvalues of $A$ where we observe that by using the preconditioner $\scaled{S}$,
PCG requires fewer iterations to achieve convergence for both instances of $B$, with the
randomised power range preconditioner $\scaledrand{S}^q$ almost achieving the same
convergence. The results here are aligned with Theorem \ref{thm:whSS_minimiser} which 
suggests that finding the correct basis for a preconditioner as well as matching the 
eigenvalues of $S$ has an impact on PCG convergence. Moreover, the Bregman divergence values in Table \ref{table:synthetic_data:divg} appear to be a useful indicator of
the PCG performance of the preconditioner, i.e., the smaller the divergence, the better
the preconditioner performs.
The inverse of $S$ is required to compute the Bregman log determinant divergence, so 
in practice the information in Table \ref{table:synthetic_data:divg} is not available.
To the best of our knowledge this divergence has not before been used to analyse 
preconditioners, so it is worthwhile developing an insight by studying smaller problems
to support the theoretical results obtained above.
For the labels $A=2$, $B=1$ means that the spectrum of $B$ decays 
exponentially while for $B=2$ it is mostly flat (cf. Figure \ref{fig:AB_spectra}).
For $B=1$ we see that $\nonscaled{S}$ (and $\nonscaledrandpower{S}$, by a small margin) 
leads to faster PCG convergence than $\scaledNys{S}$, while the opposite is true when 
$B=2$. In the latter case, the spectrum of $G$ is more easily
captured with an arbitrary sketching matrix, whereby the advantage of approximating
$G$ is more profound, suggesting perhaps an effect similar to what was observed
for the diagonal example in Section \ref{sec:spectral_analysis}.
As before, we see that the Bregman divergence values correlate with
PCG performance. Interestingly, while $\scaledrand{S}$ does not appear to be a good 
choice, our results for $\scaledrandpower{S}$  show that a few power iterations can be used to great effect.
We study the effect of sketching in more detail in Section \ref{sec:effect_low_rank}.\\

It is instructive to look at rows where label $A=3$ and $A=4$ together,
where the decay of $A$ is rapid but occurs at different points in the spectrum.
By comparing these four cases of $A$ and 
$B$, we can develop some insight into how the sketching with $\Omega$ without
any power iterations affects the results and how the randomised SVD approximations
compare to the Nystr\"om analogues. 
As before, $\scaled{S}$ and $\scaledrandpower{S}$ are superior choices whereas
$\scaledrand{S}$, $\nonscaledrand{S}$ and $\nonscaledNys{S}$ lead to the slowest 
convergence. Here, the main difference in the results is the PCG convergence using
$\scaledNys{S}$, $\nonscaled{S}$, and $\nonscaledrandpower{S}$. When the
spectrum of $A$ and $B$ is mostly flat ($A=3$, $B=2$), PCG performs best
using $\scaledNys{S}$
as a preconditioner. Since we expect the approximation error owing to randomisation
to be less pronounced, we see the benefits of the Nystr\"om approximation over the
randomised SVD preconditioner. In other words, the approximation of the $G$ matrix
is less dependent on the realisation of the sketching matrix. A similar but less 
pronounced effect is seen when $A=3$, $B=2$. This is reflected in the Bregman 
divergence, where $\scaledNys{S}$ are closer to $S$ in this nearness
measure than $\nonscaled{S}$ or its variations.\\

Overall, these results suggest that while $\scaled{S}$ is at least as good a preconditioner
as $\nonscaled{S}$ (and in all nontrivial cases, better), practical randomised numerical approximations thereof
must be carried out with great care since we see that the $\scaledrand{S}$ and $\scaledNys{S}$ can be outperformed by $\nonscaled{S}$ depending on the spectrum
of $A$ and $B$. In general, it appears that $\scaledrandpower{S}$
is the best choice when the computation of $\scaled{S}$ is not an option. 
The comparison with algebraic preconditioners showed that when the structure of
$S$ is known (i.e., $A$ is factorisable and we can make products with $B$), 
better preconditioners may be obtained by using the proposed framework. Furthermore,
the preconditioners $P_\texttt{pchol}$ and $P_\textnormal{SGS}$ require access
to the matrix $S$, which may not always be possible or convenient depending on the
application at hand. In fact, the factor $Q$ is sufficient in order to construct
a "Scaled" preconditioner, since we assume we can make products with the remainder
$S - QQ\Hermitian = B$, which lends itself better to cases where $S$ is a 
\emph{black box}. Overall, the Bregman divergence appears to be a useful tool in
suggesting how a preconditioner can perform in practice. A result such as Theorem 
\ref{thm:whSS_minimiser} supports this analysis.

\subsection{Effect of the Low Rank Approximation on the Bregman Divergence}\label{sec:effect_low_rank}

To reinforce the intuition about the Bregman divergence we conclude by
examining more closely how its terms are affected by the different low-rank
approximations described previously. For simplicity, we now only consider the
spectra given by $A=4$ and $B=1$ and set $n=100$, $m = 60$, and $r=30$. Recall
the definition in \eqref{eq:divergence}, where $X=U\diag(\lambda) U\Hermitian$
and $Y=V\diag(\theta) V\Hermitian$:
\[
\BregmanLogDet(X, Y) = \sum_{i=1}^n \sum_{j=1}^n (v_i\Hermitian u_j)^2 \Big(\frac{\lambda_i}{\theta_j} - \log\frac{\lambda_i}{\theta_j} - 1\Big).
\]

Figure \ref{fig:heatmaps_nonscaled} and \ref{fig:heatmaps_scaled} show
the terms of the Bregman divergence $\BregmanLogDet(\mathcal{P}, S)$
for each approximation $\mathcal{P}$ of $S$. The right column
in these figures show that there is little difference between the terms
$\frac{\lambda_i}{\theta_j} - \log\frac{\lambda_i}{\theta_j} - 1$ in 
the definition of the divergence. Looking at the left columns of these
figures suggests that the main contributing factor to the observed difference
in Bregman divergence values in Table \ref{table:synthetic_data:divg} lies in the
approximation of the basis for $S$ expressed by the term $(v_i\Hermitian u_j)^2$
in the expression above. In fact, we see that $\scaled{S}$ matches
the first eigenvectors of $S$ perfectly by definition, and that the first
$r$ eigenvectors of both  $\scaledrandpower{S}$ and $\scaledrand{S}$ are
almost collinear with those of $S$. This is not quite the case for
$\scaledNys{S}$, where we see the effect of the Nystr\"om approximation.
For the nonscaled preconditioners, we see in the
left column that these approximations do not capture the basis in which
the eigenvalues of $S$ are expressed, which is aligned with the PCG
results seen in the rows $A=4$ and $B=1$ of Table
\ref{table:synthetic_data:iteration}.

\begin{figure}[h!]
    \centering
    \includegraphics[scale=0.70]{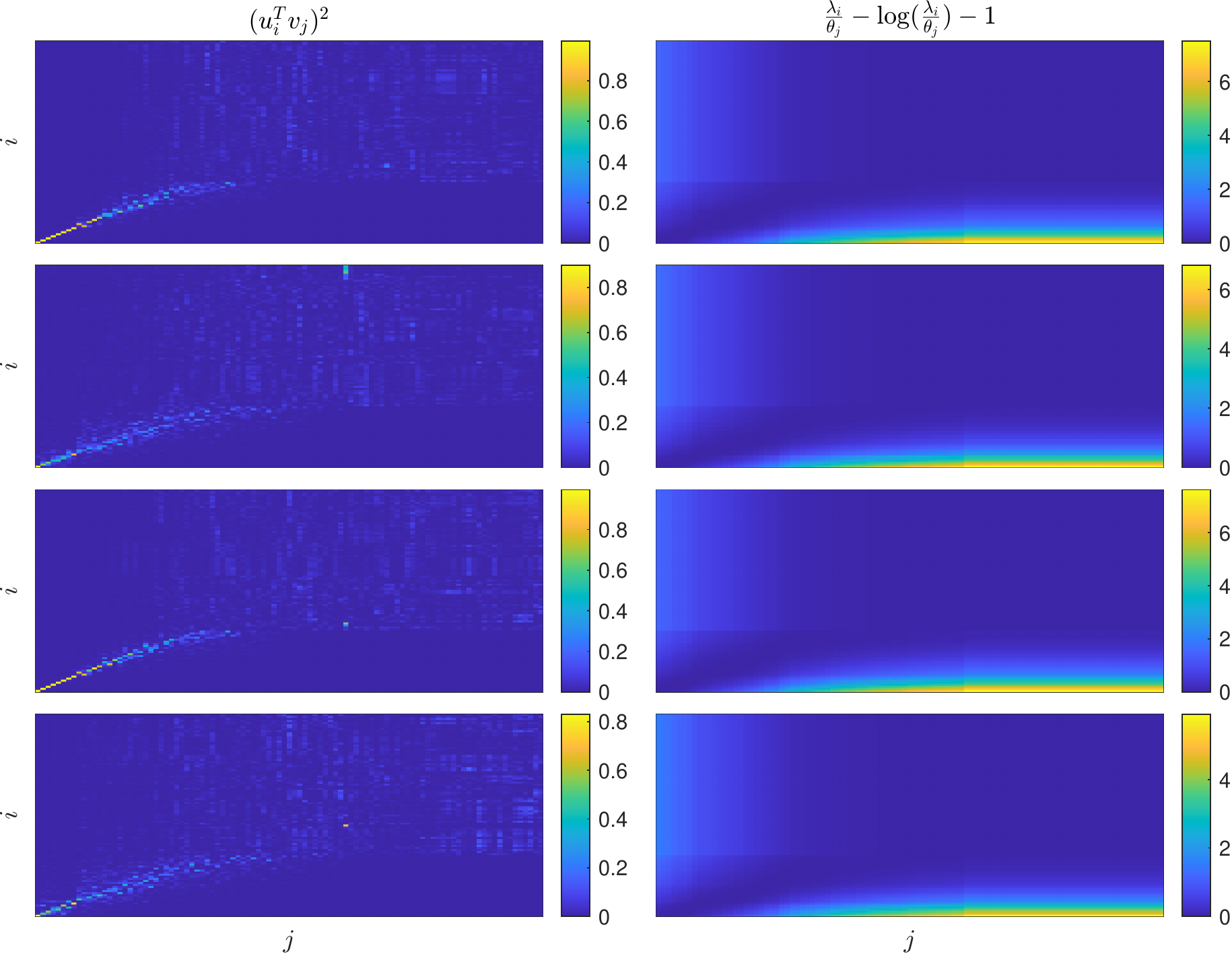}
    \caption{Bregman divergence terms for the nonscaled preconditioners. Top to bottom: $\nonscaled{S}$, $\nonscaledrand{S}$, $\nonscaledrandpower{S}$, $\nonscaledNys{S}$.}
    \label{fig:heatmaps_nonscaled}
\end{figure}

\begin{figure}[h!]
    \centering
    \includegraphics[scale=0.70]{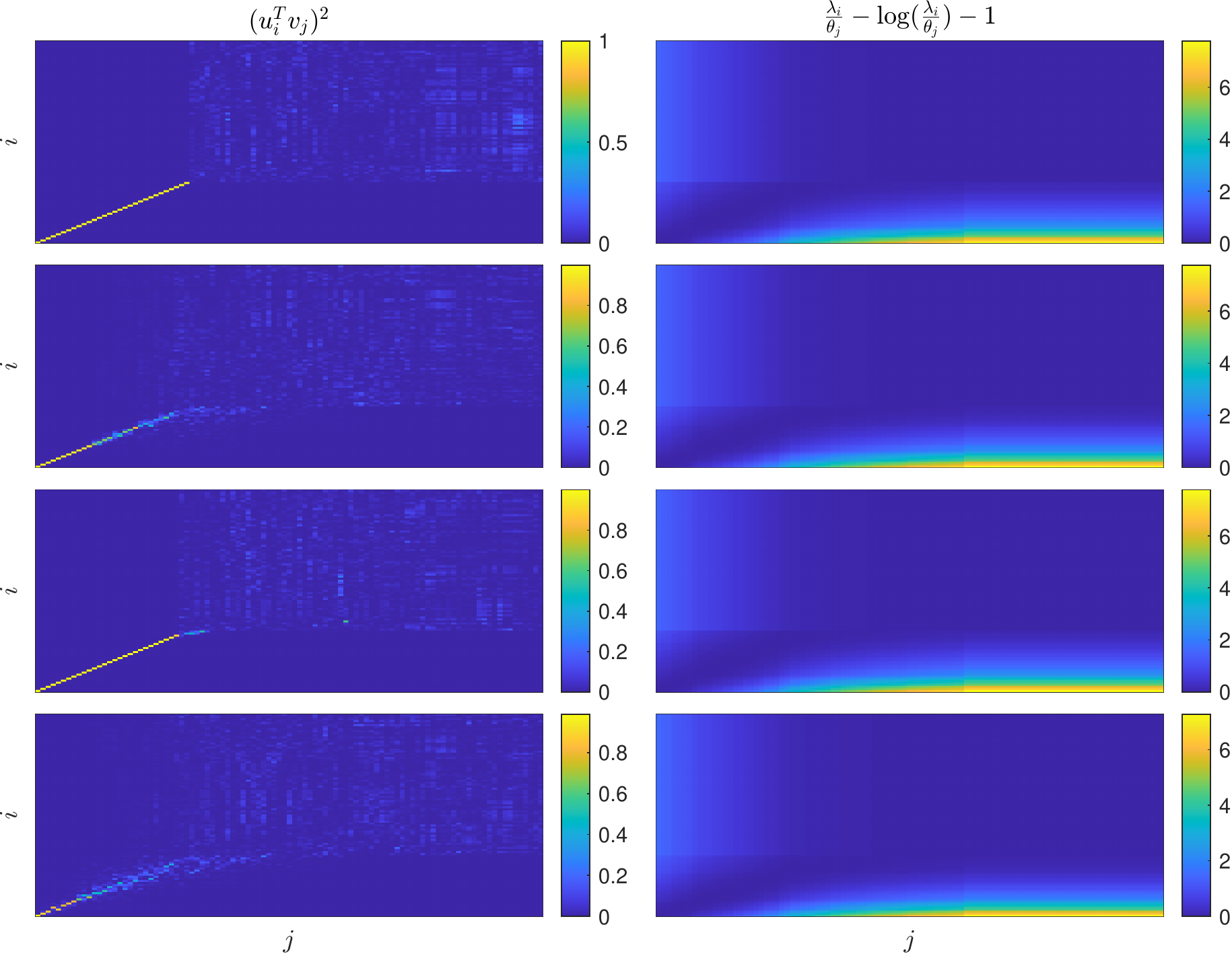}
    \caption{Bregman divergence terms for the scaled preconditioners. Top to bottom: $\scaled{S}$, $\scaledrand{S}$, $\scaledrandpower{S}$, $\scaledNys{S}$.}
    \label{fig:heatmaps_scaled}
\end{figure}

\subsection{Application to Variational Data Assimilation}\label{sec:numerical_results:var}

Data assimilation is concerned with making predictions about the future
state of a system provided noisy observations of its state along with prior
knowledge given in the form of a model. Such predictions are only useful if
they can be produced within a suitable timeframe, so efficient and scalable
methods are important for practical forecasting \cite{freitag2020numerical}. 
The objective is to assimilate $N+1$ number of state variables $x_i\in \rn$ 
at discrete times $t_0$, $\ldots$, $t_N$ given $N$ noisy observations $y_i
\in \Rm$ at times $t_1$, $\ldots$, $t_N$, so that they approximately fit a 
forward propagation model while fitting with the observations. This implies
the recurrence $x_i = \mathcal{M}_i(x_{i-1})$, $i=1,\ldots, N$, where 
$\mathcal{M}_i: \rn \rightarrow \rn$ is possibly nonlinear, subject to a
baseline criterion $x_0 = x_0^B + \xi_0$, where $\xi_0 \in \Normal(0,B)$, 
$B\in\Rnn$, for a known $x_0^B$. Therefore, $x_i$ satisfies $y_i = 
\mathcal{H}_i x_i + \epsilon_i$, where $\epsilon_i\in\Normal(0,R_i)$,
$R_i\in\Rmm$ and $\mathcal{H}_i: \rn \rightarrow \Rm$ is an (also possibly
nonlinear) observation operator. The 4D\footnote{4D since we combine 3D 
spatial information distributed in time.} variational data
assimilation problem (4D-VAR) seeks to minimise the following functional:
\begin{equation}\label{eq:4dvar}
J(x_0) = \frac12 (x_0 - x_0^B)\transp B^{-1} (x_0 - x_0^B)
+ \frac12 \sum_{i=0}^{N} (y_i - \mathcal{H}_i(x_i))\transp R^{-1} (y_i - \mathcal{H}_i(x_i)).
\end{equation}
It is commonly acknowledged that the forecasting model $\mathcal{M}_i$ contains small
errors or noise for which we do not want to overfit. We therefore 
replace the exact recurrence by $x_i = \mathcal{M}_i (x_{i-1}) + \xi_i$,
$i=1,\ldots, N$ for $\xi_i\in\Normal(0,Q_i)$, $Q_i\in\Rnn$. This motivates
the following \emph{weak constraint} formulation \cite{fisher2012weak}
 of \eqref{eq:4dvar}:
\begin{align}
J(x_0) & = \frac12 (x_0 - x_0^B)\transp B^{-1} (x_0 - x_0^B)
+ \frac12 \sum_{i=0}^{N} (y_i - \mathcal{H}_i(x_i))\transp R_i^{-1} (y_i - \mathcal{H}_i (x_i))\nonumber\\
& + \frac12 \sum_{i=1}^{N} (x_i - \mathcal{M}_i(x_{i-1}))\transp Q_i\inv (x_i - \mathcal{M}_i(x_{i-1})).\label{eq:4dvar:weak_constraint}
\end{align}
This problem poses a challenge for computational scientists, as the size
 of many relevant assimilation problems posed as \eqref{eq:4dvar:weak_constraint}
is very large.
To solve \eqref{eq:4dvar:weak_constraint} we take an incremental approach
and replace $J$ by a Gauss-Newton approximation \cite{courtier1994strategy}.
This leaves a quadratic functional $\tilde J^l$ to be minimised for each iterate
$\delta x^l = x^{l+1} - x^l$, where $l$ is the Gauss-Newton index:
\begin{align}
\tilde J^l(\delta x^l) & = (\delta x_0^l - b_0^l)\transp B^{-1} (\delta x_0^l - b_0^l)
+ \frac12 \sum_{i=0}^{N} (H^l_i \delta x_i^l - d_i^l)\transp R_i^{-1} (H_i \delta x_i^l - d_i^l)\nonumber\\
& + \frac12 \sum_{i=1}^N (M^l_i \delta x_{i-1}^l - \delta x_i^l - \eta_i^l)\transp Q_i\inv (M^l_i \delta x_{i-1}^l - \delta x_i^l - \eta_i^l),\label{eq:4dvar:weak_constraint:approx}
\end{align}
where we have used the substitutions $b_0^l = x_0^l - x_0^B$,
$d_i^l = y_i - \mathcal{H}_i(x_i^l)$, $\eta_i^l 
=x_i^l-\mathcal{M}_i(x_{i-1}^l)$ and where $H^l_i$, $M^l_i$ are the 
tangent linear approximations at iterate $l$ of $\mathcal{H}_i$, 
$\mathcal{M}_i$, respectively. We rewrite \eqref{eq:4dvar:weak_constraint:approx}
in a more compact way in \eqref{eq:4dvar:weak_constraint:approx:compact} where
we also drop the Gauss-Newton index:
\begin{align}
\tilde J(\bdx) & = (\bL\bdx-\bb)\transp\bD\inv(\bL\bdx-\bb) + (\bH\bdx-\bd)\transp\bR\inv(\bH\bdx-\bd),\label{eq:4dvar:weak_constraint:approx:compact}
\end{align}
by using the substitutions:
\begin{subequations}\label{eq:boldmatrices}
\begin{align}
& \bD= \begin{bmatrix}
B & & &\\
& Q_1 &&\\
& & \ddots \\
&  & & Q_N
\end{bmatrix},\;
\bL=\begin{bmatrix}
I & & \\
-M_1 & I &\\
& \ddots & \\
& -M_N & I
\end{bmatrix},\\
& \bR= \begin{bmatrix}
R_0 & & \\
& \ddots & \\
&  & R_N
\end{bmatrix},\;
\bH=\begin{bmatrix}
H_0 & & \\
& \ddots & \\
&  & H_N
\end{bmatrix}.
\end{align}
\end{subequations}
and:
\begin{align*}
\bb = \begin{bmatrix} b_0, \eta_1,\ldots,\eta_N \end{bmatrix}\transp,\quad \bd
= \begin{bmatrix} d_0, \ldots,d_N \end{bmatrix}\transp.
\end{align*}
Note that for $s=n(N+1)$ and $p=m(N+1)$, $\bD\textnormal{,}\,\bL\in\mathbb{R}^{s\times s}$,
$\bR\in\mathbb{R}^{p\times p}$,
$\bH\in\mathbb{R}^{p\times s}$,  $\bb\in\mathbb{R}^{s}$,
$\bd\in\mathbb{R}^p$ and we use $I$ to denote the identity matrix the
dimension of which can be inferred from the context.
Solving for the increment $\bdx$ as in
\eqref{eq:4dvar:weak_constraint:approx:compact} in called the \emph{state
formulation} (see \cite{fisher2017parallelization} for other approaches),
where the matrix that we wish to invert is given by the Hessian of $\tilde J$:
\begin{equation}\label{eq:Schur}
\bS = \bL\transp \bD\inv \bL + \bH\transp \bR\inv \bH.
\end{equation}
In practice, the $s\times s$ system matrix in \eqref{eq:Schur} can be very 
large and a good preconditioner is critical due to the small computational
budget at operational scale relative to the problem size.
In the next section we review a few key existing preconditioners for 
\eqref{eq:Schur} and describe our proposed  preconditioner based on the 
Bregman divergence framework.

\subsubsection{Preconditioning the State Formulation of 4D VAR}

Several preconditioners for \eqref{eq:Schur} have therefore been proposed.
Among these is the preconditioner
\begin{equation}\label{eq:pc:LDL}
\bL\transp \bD\inv \bL,
\end{equation}
see e.g. \cite{gratton2018guaranteeing}, which is simply the system matrix in 
\eqref{eq:Schur} without the observation term. Approximations of
$\bL\inv$ have been used in the inverse of \eqref{eq:pc:LDL} and studied in
\cite{el2015conditioning}. In this section we shall not use any approximations
of $\bL$ as the objective is to demonstrate how \eqref{eq:pc:LDL} compares
to other preconditioners that incorporate observation information in different
ways. The preconditioner \eqref{eq:pc:LDL} can be updated with a low-rank term
\cite{tabeartstein} leading to a preconditioner of the following form:
\begin{equation}\label{eq:tildeS_0}
\tilde\bS = \bL\transp\bDi\bL + \bK_r\transp \bK_r,
\end{equation}
where $\bK_r\transp \bK_r$ is a rank $r<n$ truncated SVD of $\bH\transp \bR\inv\bH$.\\

We note that $\bS$ can be written as
\begin{align*}
\bS & = \bL\transp(\bD\inv + \bL\invtransp \bH\transp \bR\inv \bH \bL\inv)\bL\nonumber\\
& = \bL\transp \bD\invhalf (I + \bDh \bL\invtransp \bH\transp \bR\inv \bH \bL\inv \bDh)\bD\invhalf \bL\nonumber\\
& = \bQ (I + \bG) \bQ\transp,
\end{align*}
where $\bQ = \bL\transp \bD\invhalf$, $\bG = \bQ\inv \bH\transp \bR\inv \bH\bQ\invtransp$.
This leads to the following preconditioner which fits into the the framework presented in 
Section \ref{sec:precond_bregman}:
\begin{equation}\label{eq:var:scaled}
\hat\bS_r = \bQ (I + \bG_r) \bQ\transp,
\end{equation}
where $\bG_r$ is a rank $r$ truncation of  $\bG$. This is also referred to
as a \emph{limited memory preconditioner} in the literature 
\cite{fisher2018low,dauvzickaite2021randomised}. Since the number
of observations is always smaller than the state space, $\rank(\bG)<n$ so by 
Theorem \ref{thm:kappa2_minimisation}, the preconditioner $\hat\bS_r$ 
minimises the condition number of the preconditioned system matrix among
preconditioners of the form
\[
\bL\transp\bD\inv\bL + \bX, \quad \rank(\bX) \leq r.
\]
As described in Section \ref{sec:rsvd}, it can be impractical to compute $\bG_r$ when
the size of the system is sufficiently large. Letting $\Omega\in\Rsr$ a Gaussian matrix of full rank we therefore define the Nystr\"om variant of both \eqref{eq:var:scaled}
and \eqref{eq:tildeS_0}:
\begin{subequations}
\begin{align}
& \scaledNys{\bS} = \bQ (I + \Nystrom{\bG}) \bQ\transp,\label{eq:var:nys}\\
& \nonscaledNys{\bS} = \bL\transp\bD\inv\bL + \Nystrom{(\bH\transp \bR\inv\bH)}.\label{eq:var:nys:nonscaled}
\end{align}
\end{subequations}
Next, we compare the PCG performance of the preconditioners \eqref{eq:pc:LDL}, 
\eqref{eq:var:nys} and \eqref{eq:var:nys:nonscaled} to an incomplete Cholesky
preconditioner (since $\bS$ will be sparse), for a specific problem instance.
 
\subsubsection{Heat Equation}

We now define the constituents \eqref{eq:boldmatrices} of the system matrix
for assimilation of the \emph{heat equation}:
\[
\frac{\partial u}{\partial t} = \frac{\partial^2 u}{\partial x^2}, \quad t \in [0,\infty), \quad x \in [0,20].
\]
We discretise the spatial operator by finite differences with $n=1000$ cells with size
$\Delta x = 2\cdot 10^{-2}$. We use a forward Euler scheme to discretise the time 
derivative with step size $\Delta t = 10^{-4}$ for a total of $N+1=100$ discrete times,
and let $r = \frac{\Delta t}{\Delta x^2}$ to define
\[
M = \begin{bmatrix}
0 & 0 & 0 & 0 & \ldots & 0 & 0\\
0 & 1-2r & r & 0 & \ldots & 0 & 0\\
0 & r & \ddots & \ddots & \ddots & \vdots & \vdots\\
\vdots & \vdots & \ddots & \ddots & \ddots & r & 0\\
0 & \ldots & 0 & 0 & r & 1-2r & 0\\
0 & \ldots & 0 & 0 & 0 & 0 & 0
\end{bmatrix}
\]
and set $M_i = M$, $i=1, \ldots, N$ to define $\bL$. We set homogeneous Dirichlet 
boundary conditions in space for all time steps. We observe half of the state vector so 
$H_i \in \mathbb{R}^{m\times n}$ has a single unit entry per row. For $i=1, \ldots, N$
we set $(H_i)_{m-j+1,j} = 1$, $j = 1,\ldots, m$ and all other elements set to zero to
define $\bH$. In practice, the covariance matrices $\bD$ and $\bR$ are typically 
ill-conditioned \cite{tabeart2022new,weston2014accounting}. To 
control the conditioning of $\bS$ we therefore construct $\bD$ by introducing a 
parameter $\tau_\bD$ and sampling $s$ points in log space from $10^{-\tau_\bD}$
to $10^{\tau_\bD}$ and setting
\begin{equation}\label{eq:bD_def}
\bD\inv = \diag(10^{-\tau_\bD}, \ldots, 10^{\tau_\bD}).
\end{equation}
We construct $\bR\inv$ in a similar fashion and set use the values 
$\tau_\bD = \tau_\bR = 1$. The parameters above lead to $\bS$ in \eqref{eq:Schur}
being a $10^5\times 10^5$ matrix with condition number $\kappa_2(\bS) = 8.1029\cdot 10^5$.\\

The purpose of this section is to compare various preconditioners for the 
Gauss-Newton iterate subproblems, so we generate $\bb\in\mathbb{R}^s$ at random
and focus on PCG for the solution of
\begin{equation}\label{eq:var:Sxb}
\bS \bdx = \bb.
\end{equation}

For the preconditioners $\tilde \bS_r$ and $\hat\bS_r$ we look at three values of $r$
of varying size relative to $n$:
\begin{equation}\label{eq:rank_variation}
r = \lfloor n \times c \rfloor, \quad c \in \{ 0.005, 0.02, 0.04\}
\end{equation}
corresponding to $0.05\%$, $2\%$ and $4\%$ of the problem size, respectively.\\

Figure \ref{fig:var:pc} shows the PCG convergence for the preconditioners above
along with the preconditioner $\bL\transp\bD\inv\bL$, an incomplete Cholesky 
preconditioner and an unpreconditioned run of the conjugate gradient metod.
Table \ref{table:var} shows the the computation time and relvative residual
achieved for these runs. 
It is clear that the preconditioner $\scaledNys{\bS}$ greatly outperforms $\nonscaledNys{\bS}$
for the same value of $r$. While increasing the rank $r$ leads to faster convergence
for both approaches, we clearly see that the scaled approach performs much better. We
even see that $\scaledNys{\bS}$ with a rank 500 truncation outperforms $\nonscaledNys{\bS}$
for a rank 4000 truncation of the observation term. This supports the theory developed
above. Neither the preconditioner 
$\bL\transp\bD\inv\bL$ or the incomplete Cholesky preconditioners are close to reaching a
small residual, as we also see from Figure \ref{fig:var:pc}. We emphasise that
there may be problem instances for which these approaches could be viable.

\begin{figure}
\centering
\includegraphics[scale=0.6]{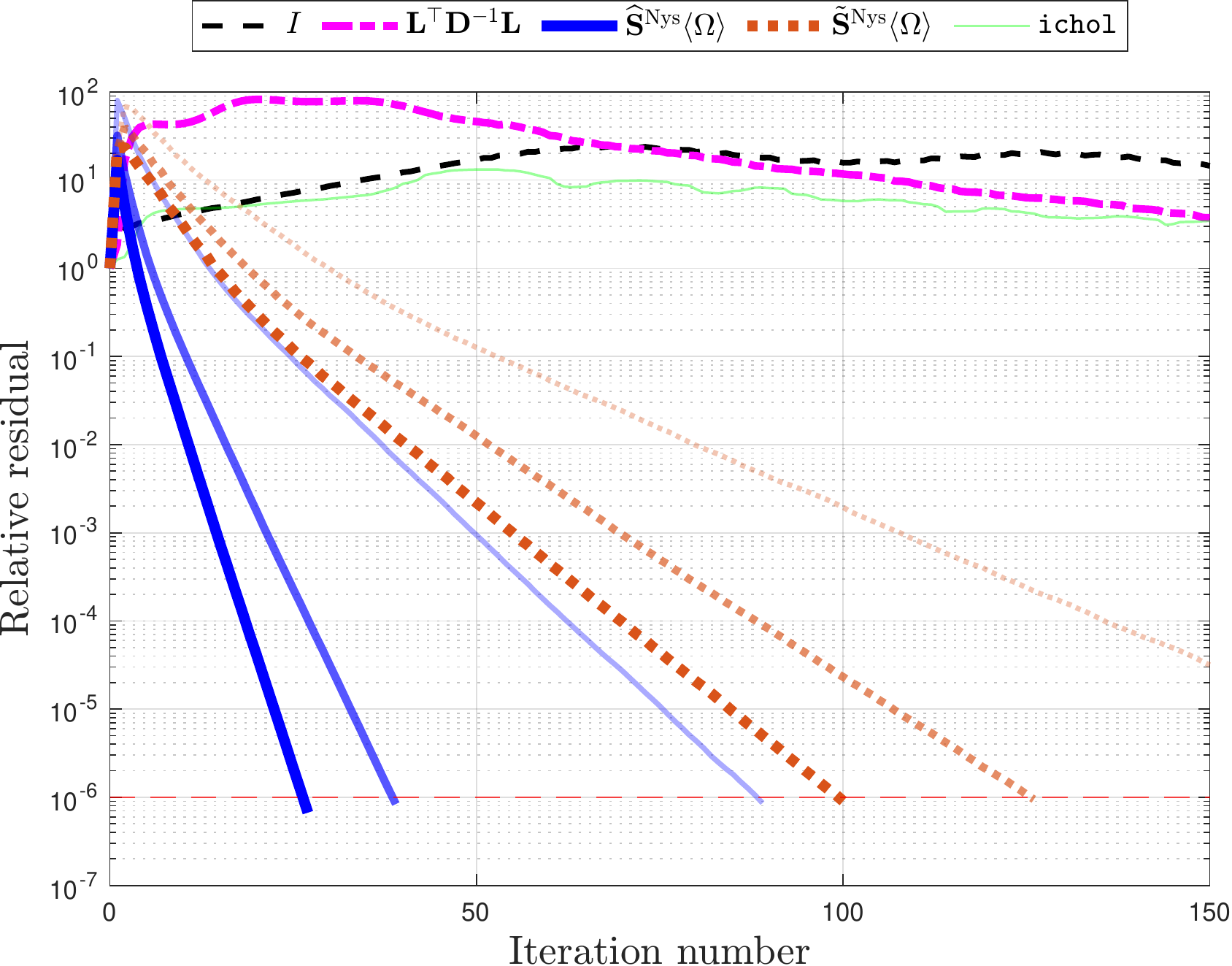}
\caption{PCG iteration count for the solution \eqref{eq:var:Sxb} for different preconditioners. 
Thicker lines correspond to higher values of the truncation rank $r$ in \eqref{eq:rank_variation}.}
\label{fig:var:pc}
\end{figure}

{
\setlength{\tabcolsep}{1.1pt}
\begin{table}[tbhp]
\footnotesize
\caption{Timings, residuals and PCG iteration counts for the solution of \eqref{eq:var:Sxb}
for various preconditioners. The residual tolerance is $10^{-6}$  and the iteration limit is $150$.}\label{table:var}
\centering
\begin{tabular*}{\textwidth}{@{\extracolsep{\fill}} l *{10}{r} }
\toprule 
\multirow{2.2}{*}{Preconditioner}  & \multirow{2.2}{*}{$I\;\;\;$} & \multirow{2.2}{*}{$\bL\transp\bD\inv\bL$} & \multirow{2.2}{*}{\texttt{ichol}} & \multicolumn{3}{c}{$\nonscaledNys{\bS}$ (eq. \eqref{eq:var:nys:nonscaled})} & \multicolumn{3}{c}{$\scaledNys{\bS}$ (eq. \eqref{eq:var:nys})}\\ \cmidrule(l{2pt}r{2pt}){5-7} \cmidrule(l{2pt}r{0pt}){8-10}
\multicolumn{4}{c}{} & $r=500$ & $r=2000$ & $r=4000$ & $r=500$ & $r=2000$ & $r=4000$ \\
\midrule
\begin{filecontents*}{results_N=100000_tex.csv}
Preconditioner,nopc,LDL,ichol,scaleda,nonscaleda,scaledb,nonscaledb,scaledc,nonscaledc
Wallclock time,3.1e-01,2.2e+01,5.3e-01,1.4e+01,2.3e+01,8.6e+00,4.8e+01,1.4e+01,6.3e+01
Residual,1.4e+01,3.8e+00,6.9e+02,8.9e-07,3.2e-05,8.1e-07,9.3e-07,6.9e-07,9.5e-07
Iterations,150,150,150,89,150,39,126,27,100
\end{filecontents*}
\csvreader[
head to column names,
]{results_N=100000_tex.csv}{}{\Preconditioner & \nopc & \LDL & \ichol & \nonscaleda & \nonscaledb & \nonscaledc & \scaleda & \scaledb & \scaledc \\}
\\[-\normalbaselineskip]\hline
\end{tabular*}
\end{table}
}

One of the limitations of our framework is that $A$ in $S=A+B$ must be known
and factorisable. In the setting above, $A = \bL\transp\bD\inv\bL$, which
may be difficult to factorise for certain forward operators and covariances.
As mentioned, the literature has explored adopting a preconditioner of the form 
\eqref{eq:pc:LDL} but substituting $\bL$ for an approximation whose inverse is
easier to compute \cite{fisher2017parallelization}. 
\cite{dauvzickaite2021time} adopts a similar methodology by approximating
$\bL\inv$ by an identity plus low-rank matrix. Both approaches result in the
construction of an approximate positive definite $\tilde A \approx A$ which is easier
to store or parallelise. As mentioned in Section \ref{sec:splitting}, we extend
the Bregman divergence framework to handle approximate factorisations $\tilde A
\approx A$ in a sequel.
\section{Summary \& Outlook}\label{sec:summary}

In this paper, we have presented several preconditioners for solving $Sx=b$ where
the matrix $S$ is a sum of $A\in\Snpp$ and $B\in\Snp$. The first proposed
preconditioner, $\scaled{S}$ in \eqref{eq:S_preconditioners:scaled}, is chosen as
a sum of a $A$ and a low-rank matrix, is the minimiser of a Bregman divergence. 
We have shown how this
preconditioner can be recovered from a scaled Frobenius norm minimisation problem,
and that, when $\rank(B)<n$, it is optimal in the sense that it minimises the condition number of the preconditioned matrix for preconditioners on its form, i.e., positive 
definite plus low rank. We have also presented variants of $\scaled{S}$ based on
randomised low-rank approximations. We also established a link between the Bregman divergence  and  the Nystr\"om approximation. The equivalence between \emph{single
pass} randomised SVD and the Nystr\"om approximation was also, to the best of the
authors' knowledge, not documented before. Our numerical experiments
illustrate how our theoretical results vary for different choices of $A$ and $B$,
and for different practical choices of low-rank approximations, i.e., randomised,
randomised with power iterations, and a randomised Nystr\"om approximation. We
also illustrated the potential of our proposed framework for preconditioning a
system stemming from a variational data assimilation problem.\\

The work in this paper offers many avenues for future research. The invariance
property of the Bregman divergence has shown to be a valuable property in
interpreting the quality of the different preconditioners studied in this work.
Knowledge of the structure of $S$ as the sum of $A$ and $B$ was essential in
utilising this invariance cf. \eqref{eq:scaled_invariance_property}.
Understanding the choice of splitting of the matrix $S$ discussed in Section
\ref{sec:splitting} could help in developing new approximation methods, for instance
when the positive definite part of $S$ is not known or easily factorisable. In
\cite{higham2019new}, the authors developed a preconditioner based on a low-accuracy
LU factorisation for ill-conditioned systems, which could yield insights in this 
direction. We could also investigate preconditioners of other forms, such
as either hierarchical matrices, leading to different constraints in a Bregman
divergence minimisation problem such as in Section \ref{sec:diag_pcs}. It could
also be fruitful to see to which extent the framework introduced here generalises
approaches such as in \cite{frangella2021randomized}. The link between
the Bregman divergence and the Nystr\"om approximation may also help to explain
why this approximation is suitable for a given
application since, as we saw in Section \ref{sec:numerical_results},
it always appears to improve PCG convergence over a randomised SVD. However,
if the computational budget of a given application allows for a randomised power range
approximation, this effort appears to be well compensated in terms of PCG performance.
Saddle-point formulations of the inner loop 4D-VAR problem have also been studied
due to their potential for parallelisation 
\cite{tabeart2021saddle,fisher2017parallelization,tabeartstein}. Low-rank approaches
have been explored in this context before \cite{freitag2018low}, so it would be 
natural to adapt the framework presented above to this setting.\\

Other divergences commonplace in information geometry could also be interesting
to analyse, e.g., the Itakura-Saito divergence or the divergence associated with the 
negative Shannon entropy. Deeper connections could be sought between this 
domain and other problems in numerical linear algebra beyond preconditioning.

\section{Acknowledgements}
This work was supported by the Novo Nordisk Foundation under grant number NNF20OC0061894. 

\appendix

\section{Synthetic Data: Generalised Eigenvalues}

In Figure \ref{fig:gevps:synt} we show the generalised eigenvalues for various
preconditioners using the experimental setup from Section \ref{sec:synthetic_data}.

\begin{figure}[h!]
\centering
    \begin{subfigure}[t]{0.45\textwidth}
        \includegraphics[scale=0.30]{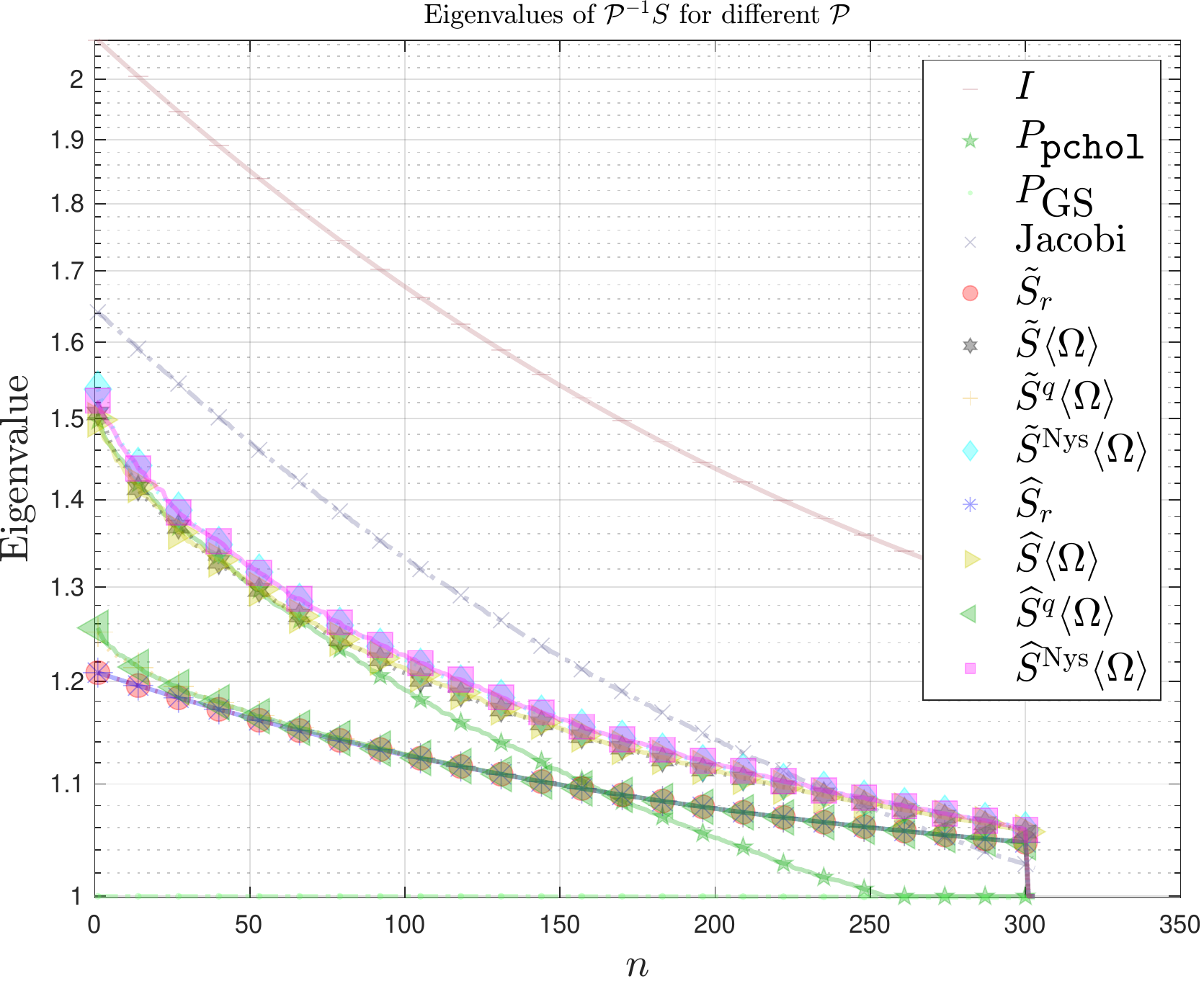}
        \caption{$A=1$, $B=1$.}
    \end{subfigure}
    \begin{subfigure}[t]{0.45\textwidth}
        \includegraphics[scale=0.30]{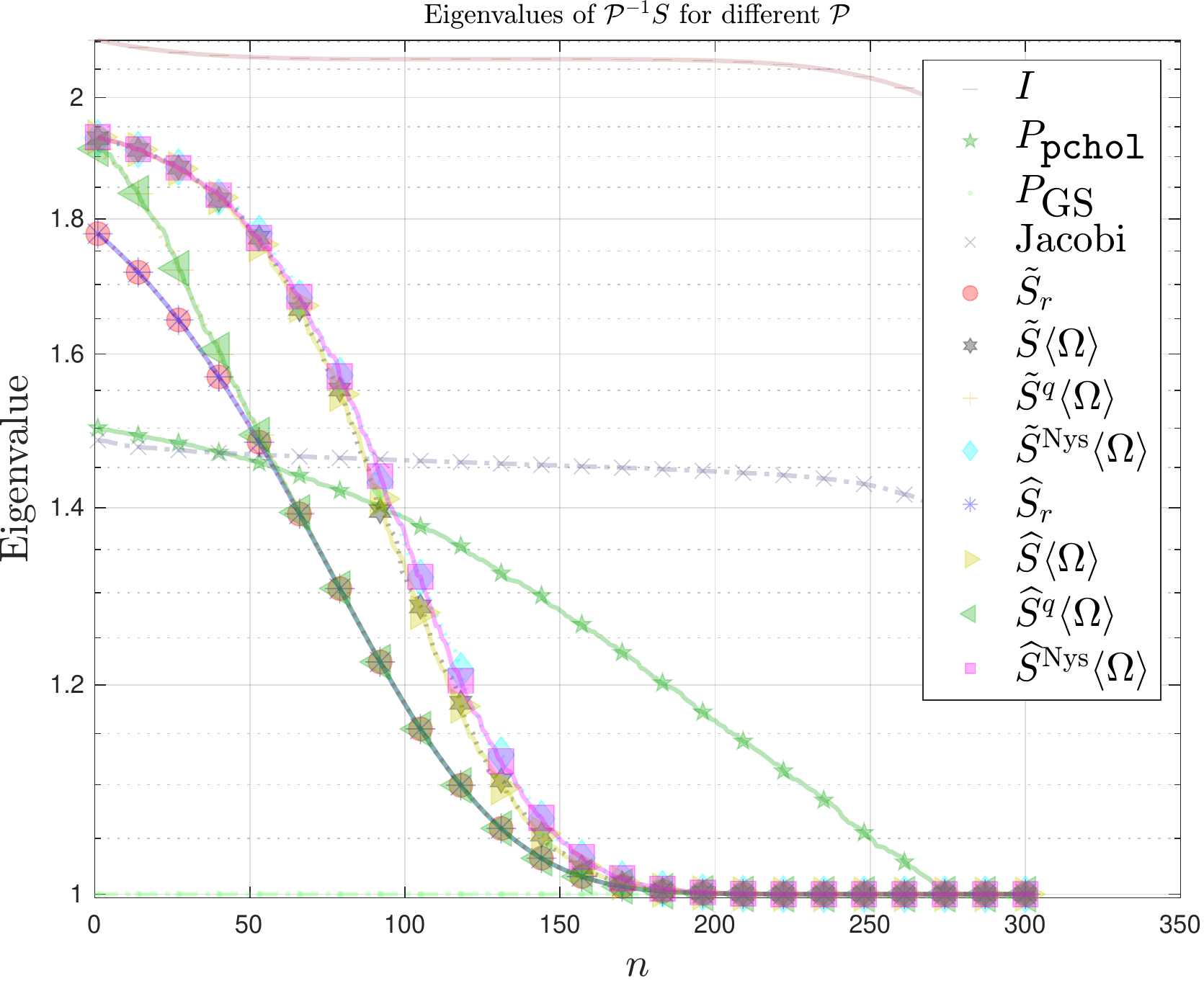}
        \caption{$A=1$, $B=2$.}
    \end{subfigure}
    \begin{subfigure}[t]{0.45\textwidth}
        \includegraphics[scale=0.30]{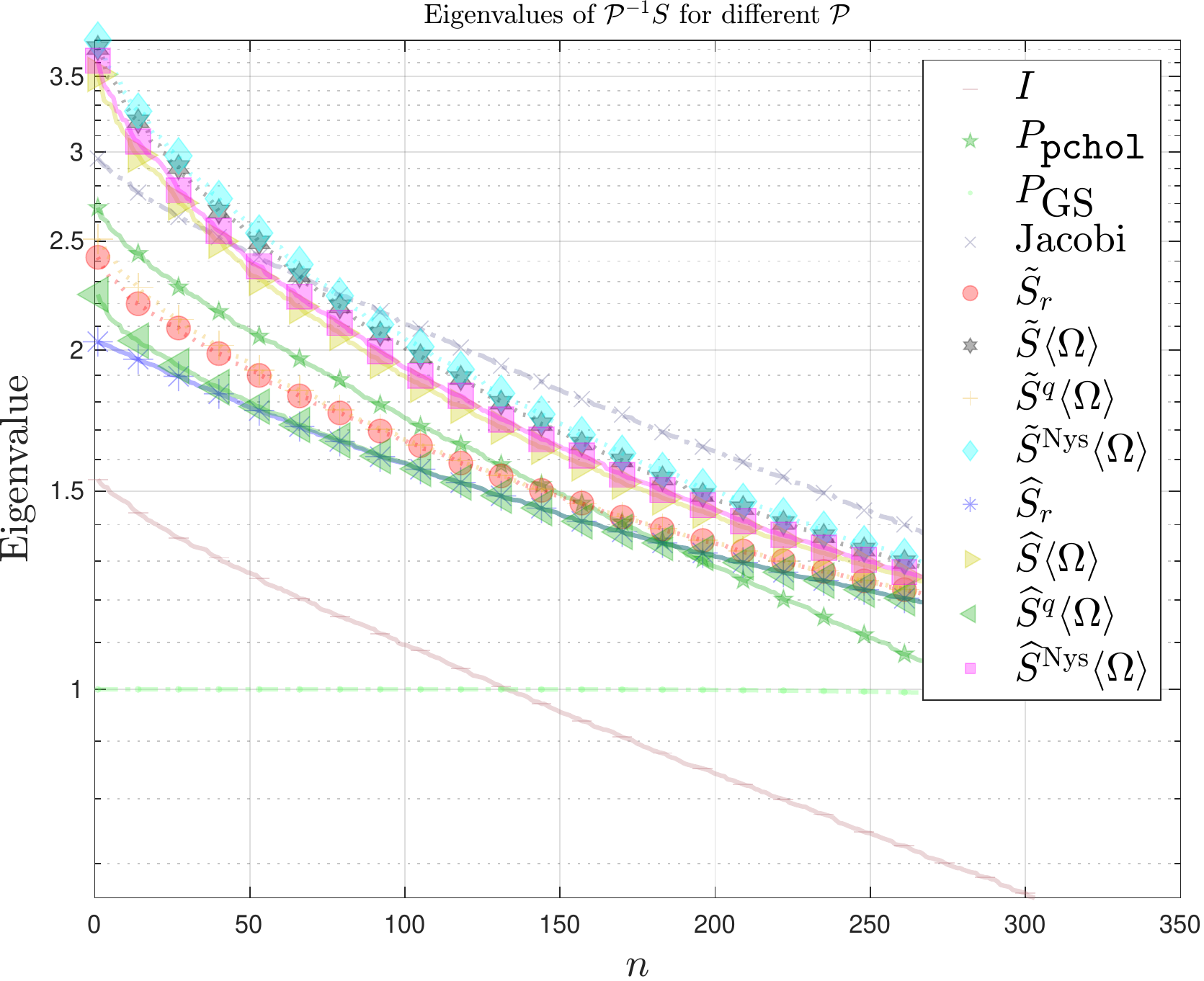}
        \caption{$A=2$, $B=1$.}
    \end{subfigure}
    \begin{subfigure}[t]{0.45\textwidth}
        \includegraphics[scale=0.30]{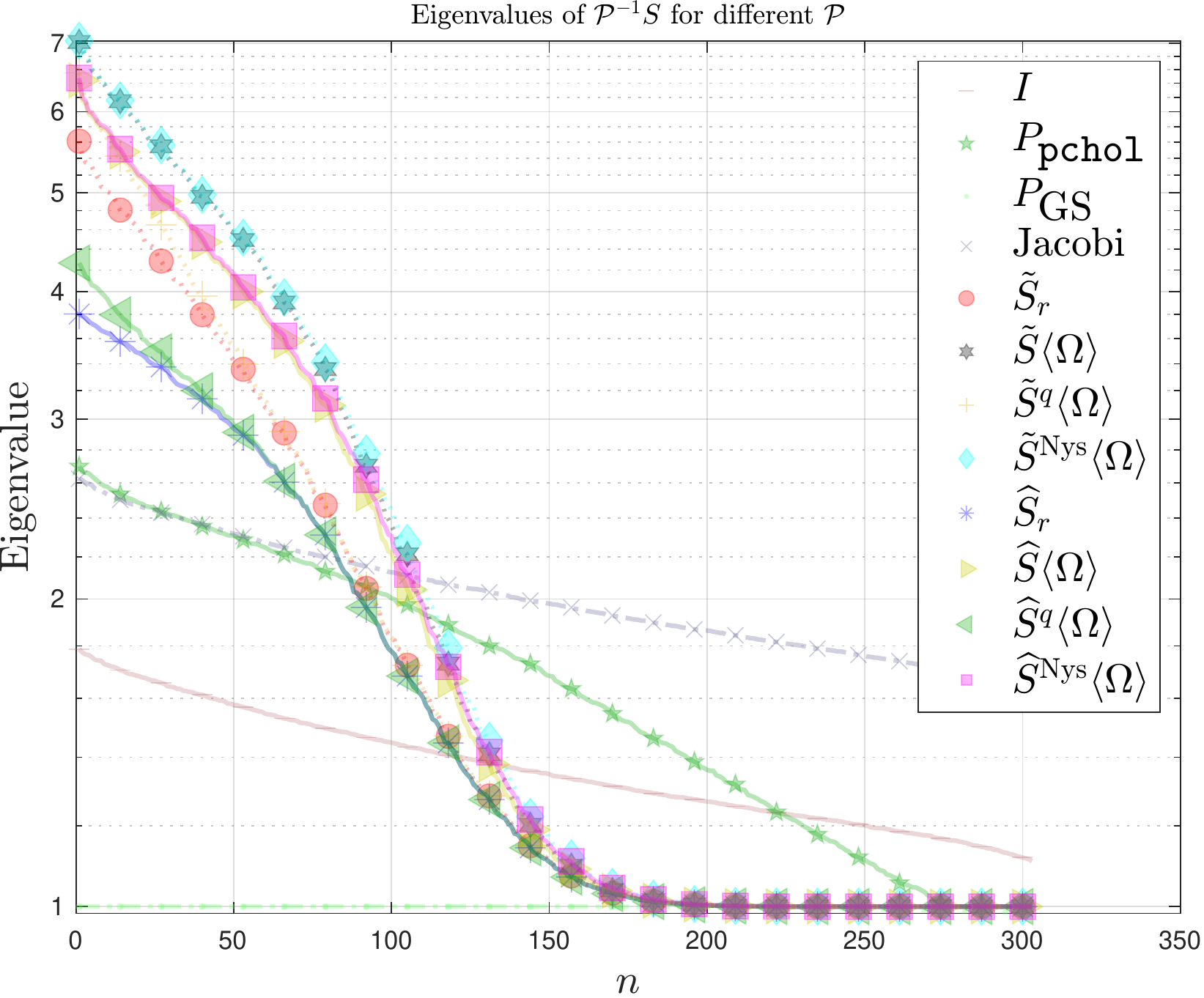}
        \caption{$A=2$, $B=2$.}
    \end{subfigure}
    \begin{subfigure}[t]{0.45\textwidth}
        \includegraphics[scale=0.30]{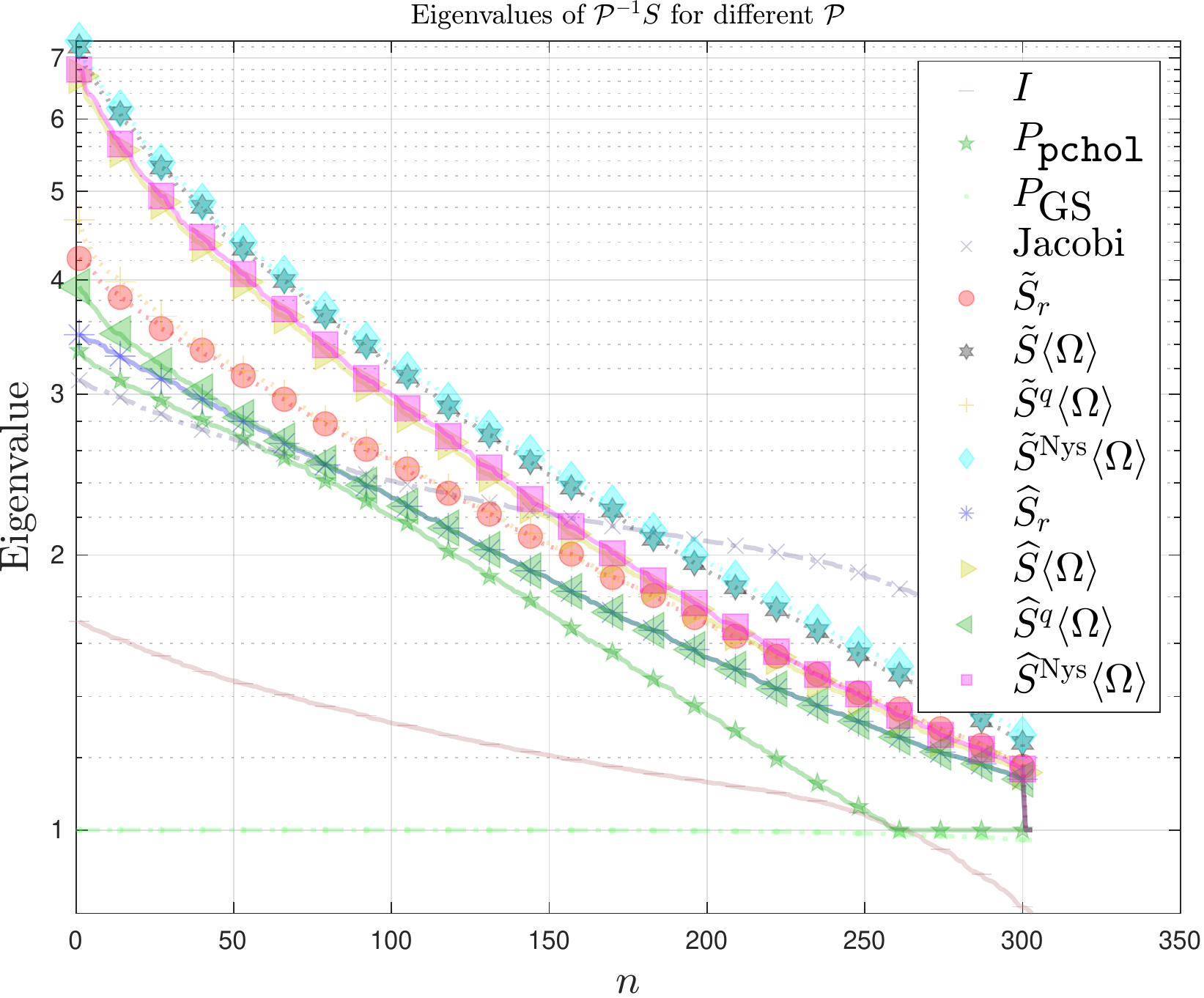}
        \caption{$A=3$, $B=2$.}
    \end{subfigure}
    \begin{subfigure}[t]{0.45\textwidth}
        \includegraphics[scale=0.30]{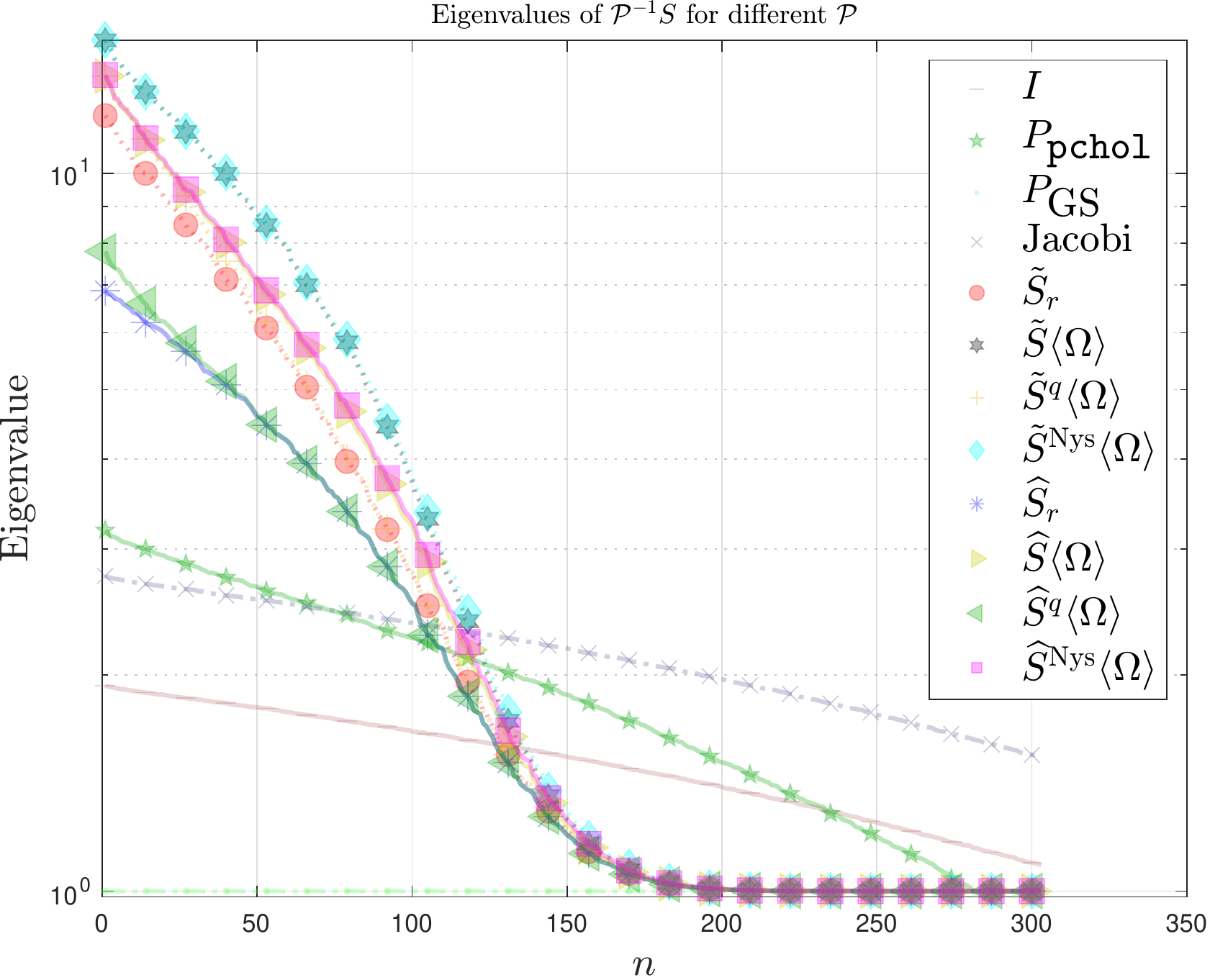}
        \caption{$A=3$, $B=1$.}
    \end{subfigure}
    \begin{subfigure}[t]{0.45\textwidth}
        \includegraphics[scale=0.30]{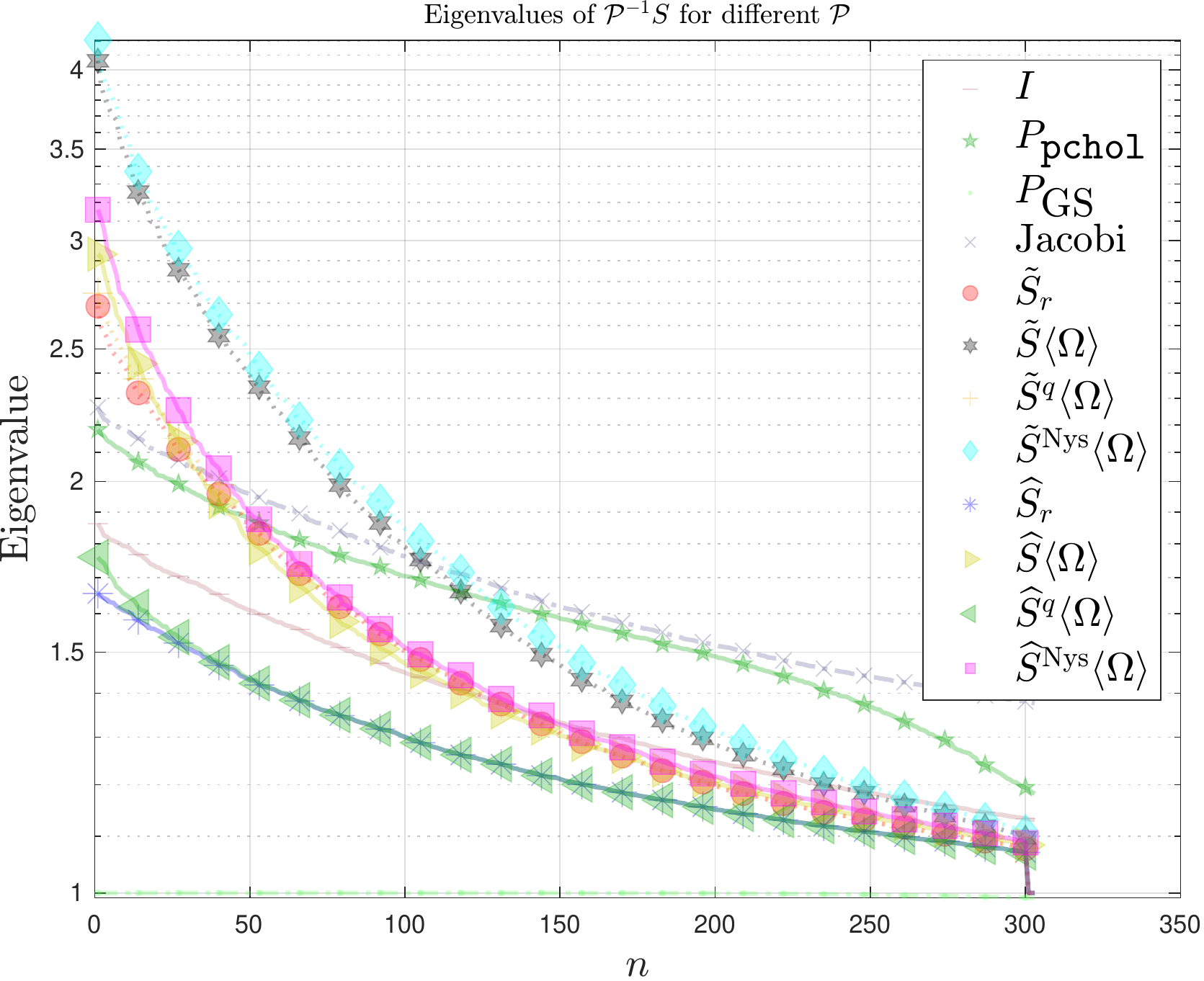}
        \caption{$A=4$, $B=1$.}
    \end{subfigure}
    \begin{subfigure}[t]{0.45\textwidth}
        \includegraphics[scale=0.30]{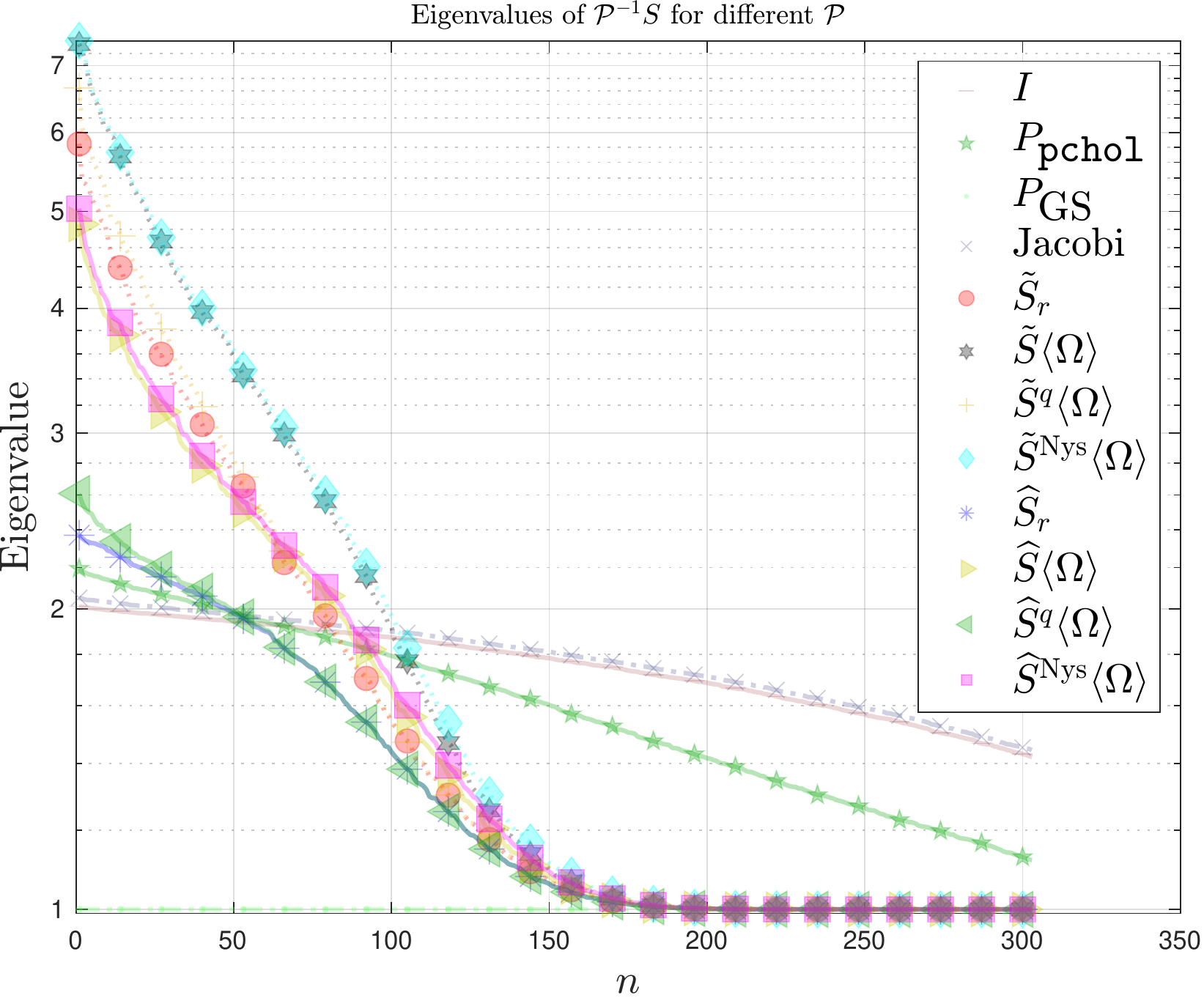}
        \caption{$A=4$, $B=2$.}
    \end{subfigure}
    \caption{Generalised eigenvalues for various preconditioners $P$. The graphs correspond to the choices of $A$ and $B$ given by the label in Table \ref{table:AB_params}.}
    \label{fig:gevps:synt}
  \end{figure}

\addcontentsline{toc}{section}{References}

\bibliographystyle{siam}
\bibliography{main}

\end{document}